\documentclass[12pt,letterpaper]{article}

\newcommand{\required}[1]{\section*{\hfil \sharp1\hfil}}

\usepackage{verbatim,url}

\usepackage{amssymb,amsmath,amsfonts,amsthm}
\usepackage{hyperref}
\usepackage{upref}
 \newtheorem{example}{Example}
  \newtheorem{prop}{Proposition}
    \newtheorem{proposition}{Proposition}
\newcommand{\Beq}{\begin{equation}}
\newcommand{\Eeq}{\end{equation}}
\newcommand{\beq}{\begin{equation*}}
\newcommand{\eeq}{\end{equation*}}
\newcommand{\bal}{\begin{align}}
\newcommand{\eal}{\end{align}}

\newcommand{\I}{\mathrm{i}}
\newcommand{\D}{\mathrm{d}}

\newcommand{\epr}{\end{proof}}

\newcommand{\lb}{\left(}

\newcommand{\vp}{\varphi}

\newcommand{\rb}{\right)}
\newcommand{\PD}{\partial}
\newcommand{\Ac}{\mathcal{A}}
\newcommand{\Bc}{\mathcal{B}}
\newcommand{\Cc}{\mathcal{C}}

\newcommand{\Ec}{\mathcal{E}}

\newcommand{\Jc}{\mathcal{J}}

\newcommand{\Nc}{\mathcal{N}}

\newcommand{\Pc}{\mathcal{P}}

\newcommand{\Rc}{\mathcal{R}}
\newcommand{\Sc}{\mathcal{S}}

\newcommand{\bpr}{\begin{proof}}

\newcommand{\Rb}{\mathbb{R}}
\newcommand{\Sb}{\mathbb{S}}

\newcommand{\g}{\boldsymbol{\gamma}}

\newcommand{\A}{\alpha}
\newcommand{\NS}{\Nc\Sc}
\newcommand{\RS}{\Rc\Sc}
\newcommand{\CS}{\Cc\Sc}
\newcommand{\B}{\beta}

\newcommand{\st}{\,:\,}
\newcommand{\wt}{\widetilde}
\newtheorem{theorem}{Theorem}
\newtheorem{lemma}{Lemma}
\newtheorem{definition}{Definition}

\theoremstyle{definition}

\newtheorem{remark}{Remark}

\newcommand{\rr}{\mathbb{R}}
\newcommand{\om}{\omega}

\newcommand{\TheTitle}{Microlocal analysis of a restricted ray transform} 
\newcommand{\TheAuthors}{Krishnan and Mishra}

\title{{{Microlocal analysis of a restricted ray transform on symmetric \lowercase{$m$}-tensor fields in $\Rb^{n}$
}}}
\date{}
\author{
  Venkateswaran P. Krishnan\thanks{TIFR Centre for Applicable Mathematics, Bangalore, Karnataka, India
 (vkrishnan@math.tifrbng.res.in)}
\and 
Rohit Kumar Mishra\thanks{Department of Mathematics, University of California, Santa Cruz, USA (rokmishr@ucsc.edu)}}
\ifpdf
\hypersetup{
  pdftitle={\TheTitle},
  pdfauthor={\TheAuthors}
}
\fi

\begin{document}
\maketitle
\begin{abstract}
We study the microlocal inversion of the ray transform on symmetric $m$-tensor fields  restricted to all lines passing through a curve in $\mathbb{R}^{n}$. From this incomplete data, we show that the wavefront set of the solenoidal component of a symmetric $m$-tensor field can be recovered modulo a known singular error term.   
\end{abstract}
\section{Introduction} Let $S^{m}=S^{m}(\Rb^{n})$ denote the space of covariant symmetric $m$-tensors in $\Rb^{n}$ and $C_{c}^{\infty}(S^{m})$ be the space of smooth compactly supported symmetric $m$-tensor fields in $\Rb^{n}$. In standard Euclidean coordinates, an element $f\in C_{c}^{\infty}(S^{m})$ can be written as 
\[
f(x)=f_{i_{1}\cdots i_{m}}(x)\D x^{i_{1}}\cdots \D x^{i_{m}},
\] with $\{f_{i_{1}\cdots i_{m}}(x)\}$ symmetric in its components, smooth and compactly supported. Here and elsewhere in the paper with repeating indices, Einstein summation convention will be assumed. 

The ray transform $\Rc$ of a symmetric $m$-tensor field $f\in C^\infty_{c}(S^m)$ is the function $\Rc f$ defined on $\Rb^{n}\times \Sb^{n-1}$ 
as   \cite{Sharafutdinov_Book}
\Beq\label{Ray transform}
	\mathcal{R}f(x,\om) = \int_{\mathbb{R}}\langle f(x+t\om), \om^{m} \rangle \D t = \int_{\mathbb{R}}f_{i_1\dots i_m}(x+t\om)\om^{i_1}\dots \om^{i_m} \D t.
\Eeq

For $\vp \in C_{c}^{\infty}(\Rb^{n}\times \Sb^{n-1})$, define the operator $\Jc \vp \in C^{\infty}(\Sb^{m})$ as \cite{Sharafutdinov_Book}
\[
\lb\Jc \vp\rb_{i_{1}\cdots i_{m}}(x)=\int\limits_{\Sb^{n-1}} \om_{i_{1}}\cdots \om_{i_{m}} \lb \int \limits_{\Rb} \vp(x-t\om,\om) \D t \rb \D S_{\om}.
\]
After a change of variable, we have 
\[
\langle If,\vp\rangle_{L^{2}(\Rb^{n}\times \Sb^{n-1})}=\langle f,\Jc \vp\rangle_{L^{2}(S^{m})},
\]
where the $L^{2}$ inner products are defined in the usual manner; see \cite{Sharafutdinov_Book}. Now, using the operator $\Jc$, the definition of $\Rc$ can be extended to compactly supported tensor field distributions \cite[\S 2.5]{Sharafutdinov_Book}.

We are interested in the inversion of the ray transform $\Rc$. It is well known 
\cite{Sharafutdinov_Book} that only the solenoidal component $f^{s}$ of a symmetric tensor field $f$ can be recovered from the transform $\Rc$. In dimensions $n\geq 3$, 
the problem of recovery of $f^{s}$ given the ray transform $\Rc f$ is over-determined. 
Therefore it is natural to study the inversion of $\Rc$ restricted to an $n$-dimensional data set. We call this problem as the incomplete data inverse problem, and address this in this paper. 

An extensive amount of literature exists in the study of scalar restricted ray transform (that is, ray transform on functions) in Euclidean as well as in Riemannian geometries \cite{Blagoveschchenskii,Mukhometov_Paper,Gelfand-Gindikin-Graev,Kirillov,Tuy,Greenleaf-Uhlmann-Duke1989, Boman-Quinto,Denisjuk1994,LanThesis,FLU,SU3,K1}. A few results exist \cite{Schuster2000,Schuster2001,Ramaseshan2004,Palamodov2009, Katsevich2013} that deal with the study of restricted Doppler transform,
and even fewer for the case of restricted ray transforms on higher order tensor fields. 
The papers \cite{Vertgeim2000,Denisjuk_Paper,Sharafutdinov2007} are the few works that we are aware of that deal with inversion of the ray transform of symmetric tensor fields  on incomplete data sets. In \cite{Vertgeim2000}, Vertgeim gave an inversion formula recovering the solenoidal component of a symmetric $m$-tensor field in $\Rb^{n}$ with the incomplete data set being all lines passing through a fixed curve satisfying a Kirillov-Tuy condition. Later in \cite{Denisjuk_Paper}, an alternate approach was given for the reconstruction (with an inversion formula) of the solenoidal component of a symmetric $m$-tensor field in $\Rb^{3}$ for the same incomplete data set as considered by Vertgeim, but the Kirillov-Tuy condition used by Denisjuk required fewer number of intersection points (see Definition \ref{K-T Defn}) compared to that of Vertgeim's. 
 In \cite{Sharafutdinov2007}, Sharafutdinov gave a slice-by-slice reconstruction procedure recovering the solenoidal component of a vector field and of a symmetric $2$-tensor field in $\Rb^{3}$ with the incomplete data set consisting of all lines parallel to a certain number of planes.  We note that the inversion formula in the work of Denisjuk \cite{Denisjuk_Paper} works only for the case of restricted ray transform on symmetric tensor fields in $\Rb^{3}$. With the Kirillov-Tuy condition used by Denisjuk, it is an interesting open question to derive an inversion procedure for the restricted ray transform on symmetric $m$-tensor fields in $\Rb^{n}$ for $n>3$.

We take a microlocal analysis approach to the study of the incomplete data tensor tomography problem. By this we mean the reconstruction of singularities of the symmetric tensor field $f$ (to the extent possible) given its ray transform. Starting with the fundamental work of Guillemin \cite{Guillemin} and  Guillemin-Sternberg \cite{Guillemin-Sternberg-AJM}, who first studied generalized Radon transform in the framework of Fourier integral operators, there have been a substantial amount of work involving microlocal analysis in the study of generalized Radon transforms \cite{GS, Greenleaf-Uhlmann-Duke1989,  Boman-Quinto-Duke, Boman1993,SU1, Katsevich2002, Lan2003,Ramaseshan2004, SU2, SU3,Uhlmann-Vasy}. The paper \cite{Greenleaf-Uhlmann-Duke1989} is a fundamental work where Greenleaf and Uhlmann studied an incomplete data ray transform on functions in the setting of Riemannian manifolds. Three related works in the direction, but in the Euclidean setting, are \cite{LanThesis,Lan2003,Ramaseshan2004}. Of these the first two deal with ray transform on functions, and the third work deals with a restricted Doppler transform on vector fields in $\Rb^{3}$, and the latter is the only result that we are aware of that deals with the microlocal analysis of a restricted ray transform on objects other than functions. 

In this article, we study the microlocal inversion of the Euclidean ray transform on symmetric $m$-tensor fields given the incomplete data set consisting of all lines passing through a fixed curve $\g$ in $\Rb^{n}$. The ray transform $\Rc$ defined in \eqref{Ray transform} restricted to lines passing through the curve $\g$ will be denoted by $\Rc_{\g}$ and its formal $L^{2}$ adjoint by $\Rc_{\g}^{*}$. We determine the extent to which the wavefront set of a symmetric $m$-tensor field can be recovered from the wavefront set of its restricted ray transform. Loosely speaking, our main result can be described as follows. Given the restricted ray transform $\Rc_{\g}$ of a symmetric tensor field $f$, we construct an approximate inverse reconstructing the solenoidal component $f^{s}$ of $f$ modulo a lower order known artifact term and a smoothing term.

The main motivation for this study comes from the works on this topic \cite{Greenleaf-Uhlmann-Duke1989,LanThesis, Lan2003,Ramaseshan2004} as already mentioned. We follow the techniques of \cite{LanThesis, Ramaseshan2004} to extend Ramaseshan's result \cite{Ramaseshan2004} to the case of restricted ray transform on symmetric $m$-tensor fields in $\Rb^{n}$. 

The article is organized as follows. In \S \ref{sect:P}, we state some fundamental results about distributions associated to two cleanly intersecting Lagrangians introduced by \cite{Melrose-Uhlmann, Guillemin-Uhlmann} and microlocal results relevant for the analysis of our transform. We then state the main result of the article. \S \ref{sect:pc} is devoted to analysis of the microlocal properties of the restricted ray transform. 
The proof of the main result is given in  \S \ref{sect:principal symbol}-\ref{sect:microlocal_inversion}. Finally in \S \ref{sect:examples}, we give some examples connected to the higher order Kirillov-Tuy condition that we use in this article.
\section{Preliminary definitions and statement of the main result}\label{sect:P}
In the first part of this section, we state the definition of distributions associated to two cleanly intersecting Lagrangians due to \cite{Melrose-Uhlmann,Guillemin-Uhlmann} and a composition calculus for a class of such distributions due to Antnoniano-Uhlmann \cite{AntonianoandUhlmann}.
In the second part, we state the main result.
\subsection{Singularities and $I^{p,l}$ classes}
	Here we give the definitions of the singularities associated with the operator $\Rc_{\g}$  as well as its associated canonical relation, and a class of
	distributions required for the analysis of the normal  operator
	$\Rc_{\g}^{*}\Rc_{\g}$.
	
	\begin{definition} \cite{GU1990b} Let $M$ and $N$ be manifolds of
		dimension $n$ and let $f:M \to N$ be $C^\infty$. Define $\Sigma= \{m
		\in M \st \det(df)_{m} = 0\}$.
	\begin{enumerate}
			\item We say $f$ drops rank \emph{simply by one} on $\Sigma$ if for each $m_{0}\in
			\Sigma$, rank~$(\D f)_{m_{0}}=n-1$ and $\D (\det(\D f)_{m_{0}})\neq 0$.
			
			\item $f$ has a \emph{Whitney fold} along $\Sigma$ if $f$ is a local
			diffeomorphism away from $\Sigma$ and $f$ drops rank simply by one on
			$\Sigma$, and $\ker \:(\D 
			f)_{m_{0}} \not \subset T_{m_{0}}\Sigma$ for every $m_{0}\in \Sigma$.

			\item $f$ is a \emph{blow-down} along $\Sigma$ if $f$ is a local
			diffeomorphism away from $\Sigma$ and $f$ drops rank by one simply on
			$\Sigma$, and $\ker
			(\D f)_{m_{0}} \subset T_{m_{0}}(\Sigma)$ for every $m_{0}\in \Sigma$.
		\end{enumerate}
	\end{definition}
We first define product-type symbols introduced in \cite{Guillemin-Uhlmann} which is then used to define $I^{p,l}$ classes of distributions.
	\begin{definition} \cite{Guillemin-Uhlmann}
		For $m \in Z^+$ and $p,l \in \mathbb{R}$, the space of product type symbols $S^{p,l}(\mathbb{R}^m;\mathbb{R}^n,\mathbb{R}^k)$ is the set of all smooth functions $a(x,\xi,\sigma)$ on $\mathbb{R}^m\times\mathbb{R}^n\times\mathbb{R}^k$ such that given $ K \subset \mathbb{R}^m$ compact, and multi-indices $\A,\B$ and $\g$, there exists a constant $C=C(K,\A,\B,\g)$ such that 
		$$ |\partial_\xi^\alpha \partial_\sigma^\beta\partial_x^\gamma a(x,\xi,\sigma) \leq C(1+|\xi|)^{p-|\alpha|}(1+|\sigma|)^{l-|\beta|}.$$
	\end{definition}
	Next we define $I^{p,l}$ classes of distributions. They were first introduced by Melrose
	and Uhlmann \cite{MU}, Guillemin and Uhlmann \cite{Guillemin-Uhlmann}
	and Greenleaf and Uhlmann \cite{GreenleafandUhlmann}.
	\begin{definition}		Two submanifolds $M$ and $N$
		intersect {\it cleanly} if $M \cap N$ is a smooth submanifold and $T(M
		\cap N)=TM \cap TN$.
		
		The pair $(M, N)$ is said to be a cleanly intersecting pair in codimension $k$ if $M$ and $N$ are cleanly intersecting and $M\cap N$ is of codimension $k$ as submanifold of both $M$ and $N$.
	\end{definition}
Consider two manifolds $X$ and $Y$ and let $\Lambda_0, \Lambda_1$, 
	and $\wt{\Lambda}_0, \wt{\Lambda}_1$ be two pairs of Lagrangian
	submanifolds of the product space $T^*X \times T^*Y$.  If $\Lambda_{0},\Lambda_{1}$ as well as $\wt{\Lambda}_{0}, \wt{\Lambda}_{1}$ 
	intersect cleanly, then  $(\tilde{\Lambda}_0, \tilde{\Lambda}_1)$ and
	$(\Lambda_0, \Lambda_1)$ are equivalent in the sense that, there is,
	microlocally, a canonical transformation $\chi$ which maps
	$(\Lambda_0, \Lambda_1)$ into $(\tilde{\Lambda}_0,
	\tilde{\Lambda}_1)$.  This leads us to the following model case.
	
	\begin{example}\label{example:basic} Let
		$\wt{\Lambda}_0=\Delta_{T^*\rr^n}=\{ (x, \xi; x, \xi)\st x \in \rr^n, \
		\xi \in \rr^n \setminus 0 \}$ be the diagonal in $T^*\rr^n \times T^*\rr^n$
		and let $\wt{\Lambda}_1= \{ (x', x'', 0, \xi''; y', x'', 0, \xi'')\st
		(x',x'') \in \rr^{k}\times \Rb^{n-k}, \ \xi'' \in \rr^{n-k} \setminus \{0\} \}$.  Then
		$\tilde{\Lambda}_0$ intersects $ \tilde{\Lambda}_1$ cleanly in
		codimension $k$ on $\Sigma=\{(x,\xi;x,\xi): \xi'=0\}$.
	\end{example}
	Since any two sets of cleanly intersecting Lagrangians are equivalent,
	we first define $I^{p,l}$ classes for the case in Example
	\ref{example:basic}.
	
	\begin{definition}\cite{Guillemin-Uhlmann}    Let $I^{p,l}(\tilde{\Lambda}_0,
		\tilde{\Lambda}_1)$ be the set of all distributions $u$ such that
		$u=u_1 + u_2$ with $u_1 \in C^{\infty}_0$ and $$u_2(x,y)=\int
		e^{\I((x'-y'-s)\cdot \xi'+(x''-y'') \cdot \xi''+ s \cdot \sigma)}
		a(x,y,s; \xi,\sigma)\D \xi \D \sigma \D s$$ with $a \in S^{p',l'}$ where
		$p'=p-\frac{n}{2}+\frac{k}{2}$ and $l'=l-\frac{k}{2}$.
	\end{definition}
Now based on the equivalence of cleanly intersecting Lagrangian pairs, the general definition of $I^{p,l}$ classes is as follows:

	\begin{definition}   \cite{Guillemin-Uhlmann} Let  $I^{p,l}( \Lambda_0, \Lambda_1)$
		be the set of all distributions $u$ such that $ u=u_1 + u_2 + \sum
		v_i$ where $u_1 \in I^{p+l}(\Lambda_0 \setminus \Lambda_1)$, $u_2 \in
		I^{p}(\Lambda_1 \setminus \Lambda_0)$, the sum $\sum v_i$ is locally
		finite and $v_i=A_{i}w_i$ where $A$ is a zeroth order FIO associated to
		$\chi_{i} ^{-1}$, the canonical transformations stated above Example \ref{example:basic}, and $w_i \in
		I^{p,l}(\tilde {\Lambda}_0, \tilde{\Lambda}_1)$.
	\end{definition}
	
	The $I^{p,l}$ class of distributions satisfies the following properties \cite{Guillemin-Uhlmann}:
	\begin{enumerate}
		\item  $ WF(u) \subset \Delta\cup \Lambda$.
			\item Microlocally, the Schwartz kernel of $u$ equals the Schwartz kernel of a pseudodifferential operator of
		order $p + l$ on $\Delta\setminus\Lambda$ and that of a classical Fourier integral operator of order $p$
		on $\Lambda\setminus\Delta$.
		\item $I^{p,l}\subset I^{p^\prime,l^\prime} \text{ if } p\leq p^\prime \text{ and } l\leq l^\prime$.
			\item  $\cap_lI^{p,l}(\Delta,\Lambda) \subset I^p(\Lambda)$.
		\item $\cap_p I^{p,l}(\Delta,\Lambda) \subset $ The class of smoothing operators.
				\item The principal symbol $\sigma_{0}(u)$ on $\Delta \setminus \Sigma$ has the singularity on $\Sigma$ as a conormal distribution of order $l-\frac{k}{2}$, where $k$ is the codimension of $\Sigma$ as a submanifold of $\Delta$ or $\Lambda$.
		\item  If the principal symbol  $\sigma_{0}(u) = 0$ on $ \Delta \setminus \Sigma$, then $ u \in I^{p,l-1}(\Delta,\Lambda)+I^{p-1,l}(\Delta,\Lambda)$.
		\item $u$ is said to be elliptic, if the principal symbol $\sigma_{0}(u)\neq 0$ on $\Delta\setminus \Sigma$ if $k\geq 2$, and for $k=1$, if $\sigma_{0}(u)\neq 0$ on each connected component of $\Delta\setminus \Sigma$.
		
	\end{enumerate}

	In the case that $\Lambda$ is a flowout, Antoniano and Uhlmann \cite{AntonianoandUhlmann} derived a composition calculus for  such  $I^{p,l}(\Delta,\Lambda)$ classes: 
	\begin{theorem}[\cite{AntonianoandUhlmann}]
		If $A \in I^{ p,l} (\Delta,\Lambda )$ and $B \in I^{ p^\prime,l^\prime} ( \Delta,\Lambda)$, then composition of $A$ and $B$,
		$A\circ B \in I^{ p+p^\prime +\frac{k}{2},l+l^\prime - \frac{k}{2}} (\Delta,\Lambda)$ and the prinicipal symbol, $\sigma_0 (A \circ B) = \sigma_0 (A)\sigma_0 (B)$, where, $k$ is the codimension of $\Sigma$ as a submanifold of either $\Delta$ or $\Lambda$.
	\end{theorem}

\subsection{Statement of the main result}\label{sect:Statement_main_result} In order to invert the restricted ray transform (restricted to lines passing through a curve), we need to place some conditions on the curve $\g$. We state them and proceed to the main result of this article. 
\begin{definition} Let $f$ be a symmetric tensor of order $m$ in $\mathbb{R}^{n-1}$. We say a collection of vectors $\{ v_1, \dots, v_{L(n,m)}\}$ in  $\mathbb{R}^{n-1}$ with $L(n,m)= {n+m-2 \choose m}$ is  generic if $f$ is uniquely determined by $f(v_1), \dots ,f(v_{L(n,m)})$. Here $f(v_i) = \langle f , v_i^m \rangle = f_{i_1\dots i_m}v_i^{i_1}\dots v_i^{i_m}$. 
\end{definition}
We note that the collection $\{v_{i}, 1\leq i\leq L(n,m)\}$ in $\Rb^{n-1}$ is generic if and only if the $m^{\mathrm{th}}$ symmetric product of the vectors $\{v_{i}: 1\leq i\leq L(n,m)\}$ denoted by $\{v_{i}^{\odot{m}},1\leq i\leq L(n,m)\}$  forms a linearly independent set.
\begin{definition}[Kirillov-Tuy condition]\label{K-T Defn} Consider a ball $B$ in $\Rb^{n}$. We say that a smooth curve $\g$ defined on a bounded interval satisfies the Kirillov-Tuy condition of order $m \geq 1$ if for almost all hyperplanes $H$ intersecting the ball $B$, there is a set of points $\g(t_1),\dots ,\g(t_{L(n,m)})$, in the intersection of the hyperplane and the curve $\g$, such that for almost all $x\in H\cap B$,  the collection of vectors $\{(x-\g(t_{i})), 1\leq i\leq L(n,m)\}$ is generic.
\end{definition}
\begin{remark}\label{Examples-remark}We make the following remarks regarding the Kirillov-Tuy condition.
	\begin{enumerate}
	\item	For $n-1=2$ and for any $m\geq 1$, the condition that a collection of  $(m+1)$-vectors $v_{1},\cdots,v_{m+1}$ are generic is equivalent to the condition that these vectors are pairwise independent. We show this in \S \ref{sect:examples}.
\item For the case of $n-1=3$ and $m=2$, the condition that a collection of  vectors $v_{1},\cdots, v_{6}$ is generic is equivalent to the condition that any collection of three vectors is linearly independent
\item 
	The situation for general $n$ and $m$ is quite different. It is possible to find $L(n,m)$ collection of vectors $v_1, \dots, v_{L(n,m)}$ in which any $(n-1)$ vectors are independent but the vectors $v_1^{\odot m}, \cdots , v_{L(n,m)}^{\odot m}$  need not be independent (see \S \ref{sect:examples} again for a counter example). 
	\end{enumerate}
\end{remark}
\begin{example}
	An example of a curve satisfying Kirillov-Tuy condition in $\Rb^{n}$ of order $m$, one can consider 
	$$ \gamma = \cup_{i=1}^{L(n,m)}l_i$$
	where $l_i = \{tv_i:t\in \mathbb{R}\setminus 0 \text{ and } v_i \in \mathbb{R}^n\}$ and $v_i \neq  v_j$ for $i\neq j$. Then almost every hyperplane in $\mathbb{R}^n$ will intersect $\gamma$ at $L(n,m)$ distinct points. And for almost every $x$ in that hyperplane,  $\{(x-\gamma_i)^{\odot m}| \text{ for }1 \leq i \leq L(n,m)\}$ forms a linearly independent set because the determinant of the matrix $\left((x-\gamma_i)^{\odot m}\right)_{i=1}^{L(n,m)}$ is a multinomial in the variable $x$ whose zero set has measure 0 in $\mathbb{R}^n$.
\end{example}
As already mentioned in the introduction, we take a microlocal analysis approach to the incomplete data tensor tomography problem. To this end, we need to put additional restrictions on the curve for our microlocal approach to work.  
	\begin{remark}[Restrictions on the curve $\g$]\label{Restrictions on the curve}
	The conditions we impose on the curve $\g$ are as follows: 
	\begin{enumerate}
		\item The curve $\g: I\to \Rb^{n}$ is smooth, regular and without self-intersections.
				\item There is a uniform bound on the number of intersection points of almost every hyperplane in $\Rb^{n}$ with the curve $\g$, see \cite{LanThesis}.
				\item The curve $\g$ satisfies the Kirillov-Tuy condition; see Definition \ref{K-T Defn}.
	\end{enumerate}
\end{remark}
Let $B$ be the ball that appears in the definition of Kirillov-Tuy condition (Definition \ref{K-T Defn}). The space of symmetric $m$-tensor fields we consider will be denoted by $\Ec'(B)$. These are symmetric $m$-tensor fields in $\Rb^{n}$ whose each component is a distribution supported in the ball $B$.

Due to the microlocal approach we adopt in this paper, we have to place certain restrictions on the wavefront set of the symmetric $m$-tensor field $f\in \Ec'(B)$, since these are the singular directions  that we can potentially recover based on this approach.
 
Following \cite{Greenleaf-Uhlmann-Duke1989, Ramaseshan2004}, we define the following sets:
\begin{alignat}{2}\label{The three sets}
	\notag  & \Xi =\Big{\{}
	(x,\xi)\in T^*B\setminus \{0\}:  \text{ there exists $L(n,m)$ directions from } x \text{ to }
	(x+\xi^\perp) \cap \g \\ & \hspace{1.6in}\text{ and any $(n-1)$ of them are linearly independent}\Big{\}}.\\
	& \Xi_{\Delta} = \Big{\{}
	(x,\xi)\in \Xi :x+\xi^\perp \text{intersects $\g$ transversely}\Big{\}}.\\
\notag 	&\Xi_{\Lambda} = \Big{\{}
	(x,\xi)\in\Xi:x+\xi^\perp \text{ is tangent (only at finite number of points) at (say)} \\
		&\hspace{1.13in} \{\g(t_{1}),\cdots, \g(t_{n})\}
	\text{ and } \langle \g^{\prime \prime}(t_{i}),\xi\rangle\neq 0 \text{ for } i=1,2,\cdots,n \Big{\}}.
\end{alignat}

\begin{theorem}\label{Main theorem}
	Let $\Xi_0 \subseteq \Xi_\Delta$ be such that $\overline{\Xi}_0 \subseteq \Xi_\Delta \cup \Xi_\Lambda$ and $K$ be a closed conic subset of \  $\Xi_0$. Let $\mathcal{E}^\prime_K(B)\subset \Ec'(B)$ denote the space of compactly supported distributions in $B$ whose wavefront set is contained in $K$. Then there exist an operator $\mathcal{B}\in I^{0,1}(\Delta,\Lambda)$ and an operator $A\in I^{-1/2}(\Lambda)$ such that for any symmetric tensor field $f$ with coefficients in $\mathcal{E}^\prime_K(B)$, we have for each $x\in \pi_{1}(K)$ (the projection to the base variable)
	\[
	\mathcal{B}\mathcal{R}_{\g}^{*}\mathcal{R}_{\g}f(x)  = f^s(x) + \mathcal{A}f(x) + \text{smoothing terms.}
	\]
\end{theorem}
\begin{remark}
		This theorem is a generalization of the results in \cite{LanThesis,Lan2003}, which dealt with the case of restricted ray transform of functions in $\Rb^{n}$ and in \cite{Ramaseshan2004} who considered a restricted ray transform of vector fields in $\Rb^{3}$. Our result can also be considered a generalization of the result in \cite{Greenleaf-Uhlmann-Duke1989}, which dealt with the ray transform on complexes in a Riemannian setting. However, we emphasize that ours is a Euclidean result, whereas \cite{Greenleaf-Uhlmann-Duke1989} studies restricted ray transforms in the very general setting of certain Riemannian manifolds. It is an interesting question to extend the current work in the setting of Riemannian manifolds as done in \cite{Greenleaf-Uhlmann-Duke1989}. 
\end{remark}
\par The strategy of proof is to compute the principal symbol of the operator $\Rc_{\g}^{*}\Rc_{\g}$ on the diagonal $\Delta$ away from the set $\Sigma$ and use this principal symbol to construct a parametrix for this operator. Since the kernel of the operator $\Rc_{\g}$ acting on symmetric $m$-tensor fields is non-trivial, we can microlocally only invert the operator $\Rc_{\g}$ to recover the solenoidal component of the tensor field. Since we are also dealing with a restricted ray transform, the inversion procedure  introduces an additional error term (in addition to smoothing terms), but this error term is a Fourier integral operator associated to the known canonical relation $\Lambda$. We begin by proving some microlocal properties of the operators $\Rc_{\g}$ and $\Rc_{\g}^{*}\Rc_{\g}$ in \S \ref{sect:pc}. \S \ref{sect:principal symbol}-\ref{sect:microlocal_inversion}  are devoted to the proof of Theorem \ref{Main theorem}.
\section{Microlocal properties of $\Rc_{\g}$ and $\Rc_{\g}^{*}\Rc_{\g}$}\label{sect:pc} This section will be devoted to some calculations regarding the singularities of the left and right projection maps from the canonical relation associated to the ray transform $\Rc_{\g}$.

We introduce spherical coordinates: 
\begin{align*}
& x_{1}=\cos\vp_{1}\\
&x_{2}=\sin\vp_{1}\cos\vp_{2}\\
&x_{3}=\sin\vp_{1}\sin\vp_{2}\cos\vp_{3}\\
& \vdots\\
&x_{n-1}=\sin\vp_{1}\sin\vp_{2}\cdots\sin\vp_{n-2}\cos\vp_{n-1}\\
&x_{n}=\sin\vp_{1}\sin\vp_{2}\cdots\sin\vp_{n-2}\sin\vp_{n-1},
\end{align*}
where $0\leq \vp_{i}\leq \pi$ for $1\leq i\leq n-2$ and $0\leq \vp_{n-1}\leq 2\pi$. For simplicity, from now on, we will denote $\vp=(\vp_{1},\cdots,\vp_{n-1})$ and the spherical coordinates defined by $\vp$ by $\vec{\vp}=(\cos \vp_{1},\cdots,\sin\vp_{1}\sin\vp_{2}\cdots\sin\vp_{n-2}\sin\vp_{n-1})$. Let us denote by $\Cc$, the line complex consisting of the  collection of all lines  passing through the curve $\g$. Any line in $\Cc$ can be described by $t,\vp$ and any point on the line parametrized by $(t,\vp)$ is of the form $\g(t)+ s \vec{\vp}$ for some $s\in \Rb$.

Let $l$ denote a line in our line complex $\Cc$ and let 
\[
Z=\{(l,x): x\in l\} \subset \Cc\times \Rb^{n}\]
be the point-line relation. We have that $(t,\vp,s)$ is a local parametrization of the set $Z$. 

We now compute the conormal bundle $N^{*}Z$ of $Z$. The tangent space $TZ$ at an arbitrary point in $Z$ determined by $(t,\vp,s)$ is spanned by the vectors: 
\[
\begin{pmatrix}
1\\\vec{0}\\\g'(t)
\end{pmatrix},
\begin{pmatrix}
0\\e_{1}\\s\PD_{\vp_{1}}\vec{\vp}
\end{pmatrix},
\cdots,\begin{pmatrix}
0\\e_{n-1}\\ s\PD_{\vp_{n-1}}\vec{\vp}
\end{pmatrix},
\begin{pmatrix}
0\\\vec{0}\\\vec{\vp}
\end{pmatrix}.
\]
Now let $(\Gamma, \xi)$ be an element of $N^{*}Z$ at an arbitrary point $(t,\vp,s)$ with $\vp_{1},\cdots,\vp_{n-2}\neq 0,\pi$. We must have 
\[
\begin{pmatrix} \Gamma\\
\xi\end{pmatrix}\cdot \begin{pmatrix}
 0\\0\\\vec{\vp}
\end{pmatrix}= \begin{pmatrix} \Gamma\\\xi\end{pmatrix}\cdot \begin{pmatrix}
1\\\vec{0}\\\g'(t)
\end{pmatrix}= \begin{pmatrix} \Gamma\\\xi\end{pmatrix}\cdot\begin{pmatrix}
0\\e_{1}\\s \partial_{\vp_{1}}
\end{pmatrix}\cdots =\begin{pmatrix} \Gamma\\\xi\end{pmatrix}\cdot\begin{pmatrix}
0\\e_{n-1}\\s \partial_{\vp_{n-1}} \vec{\vp}
\end{pmatrix}=0.
\]
From these, the fact that $\xi\cdot \vec{\vp}=0$ implies that $\xi$ belongs to the subspace generated by the unit vectors
\begin{equation}
\label{vectors vis}
 v_{1}:=\partial_{\vp_{1}}\vec{\vp},v_{2}:=\frac{1}{\sin \vp_{1}} \partial_{\vp_{2}}\vec{\vp},\cdots, v_{n-1}:=\frac{1}{\sin \vp_{1}\cdots\sin\vp_{n-2}}\partial_{\vp_{n-1}}\vec{\vp}.
\end{equation} That is $\xi$ can be written as 
\Beq\label{The vector xi}  
\xi = \sum\limits_{i=1}^{n-1}z_{i}v_{i} \mbox{ for some } z_{1},\cdots,z_{n-1}\in \Rb.
\Eeq From the remaining equations, we get, 
\Beq\label{Gamma}
\Gamma=
\begin{pmatrix}
\Gamma_{1}\\
\Gamma_{2}\\
\Gamma_{3}\\
\vdots\\
\Gamma_{n}
\end{pmatrix}=\begin{pmatrix}
-\xi \cdot \g'(t)\\ -s z_{1}\\ -s z_{2} \sin \vp_{1}\\
\vdots\\
-s z_{n-1}\sin \vp_{1}\cdots \sin\vp_{n-1}
\end{pmatrix}
\Eeq
\begin{lemma} The  map 
\[
\Phi: (t, \vp_{1},\cdots,\vp_{n-1},s,z_{1},\cdots,z_{n-1})\to (t, \vp_{1},\cdots,\vp_{n-1},\Gamma; x,\xi)
\]
with 
\begin{equation}\label{The vector Gamma}
  \Gamma=\begin{pmatrix}
-\xi \cdot \g'(t)\\ -s z_{1}\\ -s z_{2} \sin \vp_{1}\\
\vdots\\
\notag -s z_{n-1}\sin \vp_{1}\cdots \sin\vp_{n-1}
\end{pmatrix}
\end{equation}
\begin{align}
 &\notag  x=\g(t)+ s \vec{\vp}, \quad \mbox{ and } \quad \xi \mbox{ as in } \eqref{The vector xi}
\end{align}
gives a local parametrization of $N^{*}Z$ at the points where $\vp_{1},\cdots,\vp_{n-2}\neq 0,\pi$.
\end{lemma}
\bpr
We have that the pushforward $\Phi_{*}$ of $\Phi$ is given by 

\[
\Phi_{*}=
\begin{pmatrix}
&\mbox{Id}_{n\times n} & & & 0_{n\times n}\\
& \star_{n\times n} & & &\star_{n\times n}\\
&\star_{n\times n} & \vec{\vp} & &0_{n\times n-1}\\
&\star_{n\times n} & \vec{0} & v_1 & \cdots & v_{n-1}
\end{pmatrix}.
\]
Here $\vec{\vp}$ is denoted as a column vector. 
The rank of this matrix is $2n$ and hence $\Phi_{*}$ is an immersion. The map $\Phi$ is a local parametrization near points where $\vp_{1},\cdots,\vp_{n-2}\neq 0,\pi$. 
\epr
\begin{proposition}
	The operator $\Rc_{\g}$ is a Fourier integral operator of order $-1/2$ with the associated canonical relation $C$ given by $(N^{*}Z)'$ where $Z=\{(x,\g): x\in \g\}$.
	The left and the right projections $\pi_{L}$ and $\pi_{R}$ from $C$ drop rank simply by $1$ on the  set 
	\Beq\label{The set Sigma}\Sigma:=
	\{(t,\vp_{1},\cdots, \vp_{n-1}, s, z_{1},\cdots,z_{n-1}): \g'(t)\cdot \xi=0\},
	\Eeq where $\xi$ is given by \eqref{The vector xi}. The left projection $\pi_{L}$ has a blowdown singularity along $\Sigma$ and the right projection $\pi_{R}$ has a fold singularity along $\Sigma$. 
\end{proposition}
\bpr 
We consider the map $\pi_{L}$. In terms of the coordinates $(t,\vp, s,z)$ this is given by 
\[
(t,\vp_{1},\cdots,\vp_{n-1},s,z_{1},\cdots,z_{n-1})\to (t,\vp_{1},\cdots,\vp_{n-1},\Gamma),
\]
where $\Gamma$ is given above \eqref{Gamma}. We recall the vectors $v_{1},\cdots,v_{n-1}$ defined in \eqref{vectors vis}.
We have 
\[
(\pi_{L})_{*}=
\footnotesize
\begin{pmatrix}
\mbox{Id}_{n\times n} & 0_{n\times n}\\
* & \begin{pmatrix}
0 & -\g'(t)\cdot v_{1} & \cdots &\cdots & -\g'(t)\cdot v_{n-1}\\
-z_1 & -s & 0 & \cdots & 0\\
-z_2\sin \vp_1 & 0 & -s\sin\vp_1 &\cdots & 0\\
\vdots & \vdots & \vdots & \ddots & \vdots \\
-z_{n-1}\sin\vp_{1}\cdots\sin\vp_{n-2} &0 & 0 & \cdots & -s \sin\vp_{1}\cdots\sin\vp_{n-2}
\end{pmatrix}
\end{pmatrix}.
\]
The determinant of this matrix is 
\[
\det((\pi_{L})_{*})= (-1)^{n-1}s^{n-2}\sin^{n-2}\vp_{1}\sin^{n-3}\vp_{2}\cdots \sin\vp_{n-2}\g'(t)\cdot \xi.
\]
Since $s\neq 0$ and $\vp_{1},\cdots,\vp_{n-2}\neq 0,\pi$, we have that the determinant vanishes if and only if $\g'(t)\cdot \xi=0$. This is the set $\Sigma$ in \eqref{The set Sigma}. Now based on Remark \ref{Restrictions on the curve} and \eqref{The three sets} (note that these are open conditions), we need to consider only those planes $x+\xi^{\perp}$ and points on those planes such that $\g''(t)\cdot \xi \neq 0$.. 
Therefore we have that the set $\Sigma$ is an embedded submanifold of codimension $1$ in $C$. Let us show that $\pi_{L}$ has a blowdown singularity along $\Sigma$. To see this note that kernel of $(\pi_{L})_{*}$ is spanned by the vector $(0,0\cdots,0,s,-z_1,\cdots,-z_{n-1})$. Since this vector is normal to the gradient vector associated to the defining function of $\Sigma$, we have that the left projection $\pi_{L}$ has a blowdown singularity along $\Sigma$.

Now consider the canonical right projection $\pi_{R}$:  
\[
\pi_{R}: (t,\vp_{1}, \cdots, \vp_{n-1}, s, z_{1},\cdots,z_{n-1})\to (\g(t)+s\vec{\vp},\xi),
\]
where $\xi$ is as in \eqref{The vector xi}. We have 
\[
(\pi_{R})_{*}=
\begin{pmatrix}
\g'(t) & s\frac{\partial \vec{\vp}}{\PD \vp_{1}} & \cdots & s\frac{\PD \vec{\vp}}{\PD \vp_{n-1}} & \vec{\vp} & 0&\cdots &0\\
0 & \sum\limits_{i=1}^{n-1} z_{i} \frac{\PD v_{i}}{\PD \vp_{1}} &\cdots & \sum\limits_{i=1}^{n-1} z_{i} \frac{\PD v_{i}}{\PD \vp_{n-1}} &0& v_{1}& \cdots &v_{n-1}
\end{pmatrix}.
\]
Let a vector $(\A_{0},\cdots,\A_{n-1},\B_{0},\cdots,\B_{n-1})$ be in the kernel of $(\pi_{R})_{*}$. Then we must have 
\[
\begin{pmatrix}
s\frac{\partial \vec{\vp}}{\PD \vp_{1}} & \cdots & s\frac{\PD \vec{\vp}}{\PD \vp_{n-1}} & \vec{\vp}
\end{pmatrix}
\begin{pmatrix}
\A_{1}\\
\vdots \\
\A_{n-1}\\
\B_{0}
\end{pmatrix}=-\A_{0}\begin{pmatrix}
\g_{1}'(t)\\
\vdots\\
\g_{n}'(t)
\end{pmatrix}.
\]
Choosing $A_{0}=-s \neq 0$, we have that 
\[
\begin{pmatrix}
\A_{1}\\
\vdots \\
\A_{n-1}\\
\end{pmatrix}=
\begin{pmatrix}
v_{1}\cdot \g'(t)\\
\frac{1}{\sin \vp_{1}} v_{2}\cdot \g'(t)\\
\vdots\\
\frac{1}{\sin \vp_{1}\cdots \sin \vp_{n-2}}v_{n-1}\cdot \g'(t)
\end{pmatrix}.
\]
We also must have 
\[
\sum\limits_{i=1}^{n-1} \B_{i} v_{i}=-\sum\limits_{j=1}^{n-1} \sum\limits_{i=1}^{n-1} z_{i} \frac{\PD v_{i}}{\PD \vp_{j}}\A_{j}.
\]
Hence
\Beq\label{fold equality}
\lb \sum\limits_{i=1}^{n-1} \B_{i} v_{i}\rb\cdot \g'(t)=-\lb \sum\limits_{j=1}^{n-1} \sum\limits_{i=1}^{n-1} z_{i} \frac{\PD v_{i}}{\PD \vp_{j}}\A_{j}\rb \cdot \g'(t).
\Eeq
Now let us consider the dot product: 
\[
\begin{pmatrix}
\A_{0}\\
\A_{1}\\
\vdots\\
\A_{n-1}\\
\B_{0}\\
\B_{1}\\
\vdots\\
\B_{n-1} 
\end{pmatrix}
\mbox{ and } 
\begin{pmatrix}
\g''(t)\cdot \xi\\
\g'(t)\cdot \sum\limits_{i=1}^{n-1} z_{i} \frac{\PD v_{i}}{\PD \vp_{1}}\\
\vdots\\
\g'(t) \cdot \sum\limits_{i=1}^{n-1} z_{i} \frac{\PD v_{i}}{\PD \vp_{n-1}}\\
0\\
\g'(t) \cdot v_{1}\\
\vdots \\
\g'(t)\cdot v_{n-1}
\end{pmatrix}.
\]
Note that the second vector is normal to the tangent space of the submanifold $\Sigma$ defined by $\{\g'(t) \cdot \xi=0\}$.

We have from \eqref{fold equality} that 
\[
\sum\limits_{j=1}^{n-1} \A_{j} \lb \g'(t)\cdot \sum\limits_{i=1}^{n-1} z_{i} \frac{\PD v_{i}}{\PD \vp_{j}}\rb + \g'(t) \cdot \lb \sum\limits_{i=1}^{n-1} \B_{i}v_{i}\rb=0.
\]
Recall that we chose ${\A_{0}=-s}\neq 0$. Hence we have that the right projection $\pi_{R}$ has a fold singularity along $\Sigma$. Finally, the fact that $\Rc_{\g}$ is a Fourier integral operator of order $-1/2$ follows from the general theory of generalized Radon transforms \cite{Guillemin,Guillemin-Sternberg_Book}.
\epr
\begin{lemma}
The wavefront set of the Schwartz kernel of $\Rc_{\g}^{*}\Rc_{\g}$ satisfies the following: 
\[
WF(\Rc_{\g}^{*}\Rc_{\g})\subset \Delta \cup \Lambda,
\]
where $\Delta$ is the diagonal Lagrangian 
\[
\Delta =\left\{(x,\xi; x,\xi): x= \g(t) + s\theta, \xi\in \theta^{\perp}\setminus \{0\}\right\}
\]
and 
\begin{align*}
	\Lambda =\left\{(x,\xi,y,\frac{\tau}{\tilde \tau}\xi):x=\g(t)+\tau\theta,y=\g(t)+\tilde\tau\theta,\xi \in \theta^\perp\setminus \{0\}, \g'(t)\cdot\xi=0, \tau\neq 0,\tau' \neq 0 \right\}.
\end{align*}
\end{lemma}
\bpr By the H\"ormander-Sato Lemma, we have that 
\begin{align*}
WF(\Rc_{\g}^{*}\Rc_{\g})\subset \{&(x,\xi,y,\eta): (x,\xi, t,\vp,\Gamma)\in WF(\Rc_{\g}^{*}) \mbox{ and } (\wt{t},\wt{\vp},\wt{\Gamma},y,\eta)\in WF(\Rc_{\g})\\
&\mbox{ and } \wt{t}=t, \wt{\vp}=\vp, \wt{\Gamma}=\Gamma \}.
\end{align*}
Now since $\Gamma=\wt{\Gamma}$, we have
\begin{align*}
& \g'(t)\cdot \xi= \g'(t)\cdot \eta\\
& sz_{i}=\wt{s}\wt{z_{i}} \mbox{ for all } 1\leq i\leq n-1.
\end{align*}
Since $s$ and $\wt{s}$ are non-zero, we have 
\[
\wt{z_{i}}=\frac{s}{\wt{s}} z_{i} \mbox{ for all } 1\leq i\leq n-1.
\]
Now 
\[
\eta = \sum\limits_{i=1}^{n-1} \wt{z_{i}} v_{i}=\frac{s}{\wt{s}}\xi.
\]
Hence 
\[
\g'(t) \cdot \xi= \frac{s}{\wt{s}}\g'(t)\cdot \xi.
\]
If $s\neq \wt{s}$, we have that $\g'(t)\cdot \xi=0$. On the other hand if $s=\wt{s}$, then the wavefront set then there is a contribution in $\Delta$ and if $\tau \neq \wt{\tau}$, there is a contribution in $\Lambda$.

Now we have that  $\Delta$ and $\Lambda$ intersects cleanly. For,
\[
\Delta \cap \Lambda =\{(\g(t)+s\vp,\xi; \g(t)+s\vp,\xi): \vp\in \Sb^{n-1}, s>0, \xi \in \vp^{\perp}\setminus\{0\}, \g'(t)\cdot \xi=0\}.
\]
Since $\g''(t)\cdot \xi\neq 0$, we have that $\Delta\cap \Lambda$ is an embedded submanifold of codimension $1$ in either $\Delta$ or $\Lambda$, and the intersection is clean.
\epr
\begin{lemma} \cite{LanThesis} The Lagrangian $\Lambda$ arises as a flowout from the set $\pi_{R}(\Sigma)$.
\end{lemma}
\section{The principal symbol of the operator $\Rc_{\g}^{*}\Rc_{\g}$}\label{sect:principal symbol}

In this section, following the techniques of \cite{LanThesis, Ramaseshan2004}, we compute the principal symbol of the 
Schwartz kernel of the operator 
\[
A^{i_{1}\cdots i_{m}j_{1}\cdots j_{m}}=A= \Rc_{\g}^{*}\omega^{i_1}\dots\omega^{i_m}\omega^{j_1}\dots\omega^{j_m} \Rc_{\g}
\]
 on the diagonal Lagrangian $\Delta$ away from the set $\Sigma$, that is, on $\Delta \setminus \Sigma$. In order to do this, we use  \cite[Definition 4.1.1]{Duistermaat}. We compute the  leading order term as $\lambda\rightarrow \infty$ of $e^{i\lambda\psi}(A(\chi e^{-i\lambda\psi}))$ where $\chi$ and $\psi$ are suitably chosen cutoff and phase functions respectively.  We use the method of stationary phase to compute this leading order term.  Let $w_0 = (x_0,\xi, x_0,\xi)\in \Delta \setminus \Sigma$. The phase function $\psi(x,y,w_{0})$ is to be chosen such that the graph $\Gamma$ of $(x,y) \mapsto \nabla_{x,y}\psi(x,y,w_{0})$ intersects the diagonal $\Delta$ transversely at $w_0$ (see \cite[Definition 4.1.1]{Duistermaat}), where $\Delta=\{(x,\eta, x,\eta); x\in \Rb^{n}, \eta\in \Rb^{n}\setminus \{0\}\}$.

Following \cite{Lan2003} (see also \cite{Ramaseshan2004}), we define $\psi$ as follows:
$$ \psi(x,y,w_{0}) = \langle x-x_0,\xi \rangle +\langle y-x_0,-\xi \rangle +\frac{1}{2}|x-x_0|^2h(\xi,-\xi) +\frac{1}{2}|y-x_0|^2k(\xi,-\xi)$$
where $h$ and $k$ are homogeneous of degree $1$, non-negative and do not vanish simultaneously. Then we have 
\[
\Gamma=\{x,y, \xi+(x-x_{0})h(\xi,-\xi), -\xi+(y-x_{0})k(\xi,-\xi): (x,y)\in \Rb^{n}\times \Rb^{n}\}.
\]
 
To prove the transversality condition, we calculate the tangent space $T_{w_0}(\Delta)$ and $T_{w_0}(\Gamma)$.

We have that $T_{w_0}(\Delta)$ is spanned by the columns of the $4n\times 2n$ matrix
\begin{align*}
	\begin{pmatrix}
		I_n & 0\\
		I_n & 0\\
		0 & I_n\\
		0 & I_n
	\end{pmatrix}
\end{align*}
and $T_{w_0}(\Gamma)$ is spanned by the columns of $4n \times 2n$ matrix
\begin{align*}
	\begin{pmatrix}
		I_n & 0\\
		0 & I_n \\
		d^2_{xx}\psi & d^2_{xy}\psi\\
		d^2_{xy}\psi & d^2_{yy}\psi
	\end{pmatrix}=\begin{pmatrix}
		I_n & 0\\
		0 & I_n \\
		hI_n & 0\\
		0 & kI_n
	\end{pmatrix}.
\end{align*} 
Consider the combined $4n \times 4n$ matrix and apply column and row operations
\begin{align*}
	&\begin{pmatrix}
		I_n & 0 & I_n & 0\\
		I_n & 0 & 0 & I_n\\
		0 & I_n & hI_n & 0\\
		0 & -I_n & 0 & kI_n\\
	\end{pmatrix}
	\sim  \begin{pmatrix}
		0 & 0 & I_n & 0\\
		0 & 0 & 0 & I_n\\
		-hI_n & I_n & hI_n & 0\\
		-kI_n & -I_n & 0 & kI_n\\
	\end{pmatrix}
	\sim  \begin{pmatrix}
		0 & 0 & I_n & 0\\
		0 & 0 & 0 & I_n\\
		0 & I_n & hI_n & 0\\
		-(h+k)I_n & -I_n & 0 & kI_n\\
	\end{pmatrix}\\
	&\sim  \begin{pmatrix}
		0 & 0 & I_n & 0\\
		0 & 0 & 0 & I_n\\
		0 & I_n & 0 & 0\\
		-(h+k)I_n & -I_n & 0 & 0\\
	\end{pmatrix}
	\sim  \begin{pmatrix}
		0 & 0 & I_n & 0\\
		0 & 0 & 0 & I_n\\
		0 & I_n & 0 & 0\\
		-(h+k)I_n & 0 & 0 & 0
	\end{pmatrix}.
\end{align*}
The combined matrix has rank $4n$ and hence the transversality condition is satisfied for our choice of $\psi$.

\par Let $K_A$ be the Schwartz kernel of $A$. For test functions $\phi_1(x)$ and $\phi_2(y)$ supported near $x_0$, we have
\begin{align*}
	\notag & \langle K_A,e^{-i\psi(x,y,w_0)}\phi_1\phi_2\rangle
	= \\
	&\langle\omega^{i_1}\cdots\omega^{i_m}\Rc_{\g}(e^{-i(\langle y-x_0,-\xi \rangle+\frac{1}{2}|y-x_0|^2k(\xi,-\xi)}\phi_{2}),
	\omega^{j_1}\cdots\omega^{j_m}\Rc_{\g}( e^{-i(\langle x-x_0,\xi \rangle+\frac{1}{2}|x-x_0|^2h(\xi,-\xi)}\phi_{1})\rangle\\
	&=\int \omega^{i_1}\cdots\omega^{i_m}\omega^{j_1}\cdots\omega^{j_m}\phi_1(\g(t)+s_1\omega)\phi_2(\g(t)+s_2\omega)e^{-i\lambda\tilde\psi}\D s_1\D s_2\D t\D \omega\\
	\notag &=\int e^{-i\lambda\tilde\psi}\omega^{i_1}(\theta)\cdots\omega^{i_m}(\theta)\omega^{j_1}(\theta)\cdots\omega^{j_m}(\theta)\phi_1(\g(t)+s_1\omega(\theta))\phi_2(\g(t)+s_2\omega(\theta)) \\
	& \hspace{0.5in}\prod\limits_{k=1}^{n-2} \lb \sin \theta_{k}\rb^{n-1-k}\D s_1\D s_2\D t \D \theta_{1}\cdots \D \theta_{n-1},
\end{align*}
where
\begin{align*}
	\tilde \psi(s_1,s_2,t,\theta_1,\theta_2,\dots,\theta_{n-1}) &= (s_1-s_2)\omega\cdot\xi_0+\frac{1}{2}|\g(t)+s_1\omega(\theta)-x_0|^2h(\xi_0,-\xi_0)\\
	&\quad +\frac{1}{2}|\g(t)+s_2\omega(\theta)-x_0|^2k(\xi_0,-\xi_0),
\end{align*}
$|\xi| = \lambda,\ \xi_0=\frac{\xi}{\lambda},\ \omega= (\omega^1,\omega^2,\dots,\omega^n)$ with
\begin{align*}
	\omega^1 &= \cos \theta_{1}\\
	\omega^{2} &=\sin \theta_{1}\cos \theta_{2}\\
	\omega^{3}&= \sin \theta_{1} \sin \theta_{2}\cos \theta_{3}\\
	\vdots\\
	\omega^{n}&=\sin\theta_1\sin\theta_2\dots \sin\theta_{n-2}\sin\theta_{n-1}
\end{align*}
with $\theta_i \in (0,\pi)$ for $i=1,2,3,\dots,(n-2)$ and $\theta_{n-1}\in [0,2\pi]$.

We apply the method of stationary phase to the above integral. We compute the critical points of $\tilde\psi$ as a function of $s_1,s_2,t,\theta_1,\dots,\theta_{n-2}$ and $\theta_{n-1}$. We have 
\begin{align*}
	&\wt\psi_{s_1} = \langle\omega,(\g(t)+s_1\omega-x_0)h+\xi_0\rangle, \\
	&\wt\psi_{s_2} = \langle\omega,(\g(t)+s_2\omega-x_0)k-\xi_0\rangle ,\\
	&\wt\psi_t = \langle \g^\prime(t),(\g(t)+s_1\omega-x_0)h+(\g(t)+s_2\omega-x_0)k\rangle,\\
	&\wt\psi_{\theta_1} = \langle \omega_{\theta_1},(s_1-s_2)\xi_0+s_1(\g(t)+s_1\omega-x_0)h+s_2(\g(t)+s_2\omega-x_0)k \rangle,\\
	&\wt\psi_{\theta_2} = \langle \omega_{\theta_2},(s_1-s_2)\xi_0+s_1(\g(t)+s_1\omega-x_0)h+s_2(\g(t)+s_2\omega-x_0)k \rangle,\\
	&\wt\psi_{\theta_3} = \langle \omega_{\theta_3},(s_1-s_2)\xi_0+s_1(\g(t)+s_1\omega-x_0)h+s_2(\g(t)+s_2\omega-x_0)k \rangle,\\
	& \hspace{2in}\vdots \\
	&\wt\psi_{\theta_{n-1}} = \langle \omega_{\theta_{n-1}},(s_1-s_2)\xi_0+s_1(\g(t)+s_1\omega-x_0)h+s_2(\g(t)+s_2\omega-x_0)k \rangle.
	\end{align*}
	Now multiplying $\wt{\psi}_{s_{1}}$ by $k$ and $\wt{\psi}_{s_{2}}$ by $h$, subtracting and setting it to $0$, we get, 
	\begin{align*}
	k\tilde\psi_{s_1}-h\tilde\psi_{s_2} &= 0 \mbox{ implies }  \langle\omega,\omega(s_1-s_2)hk +(h+k)\xi_0\rangle = 0.
	\end{align*}
	From this we have 
	\begin{align}\label{Omega-xi0}
	\langle\omega,\xi_0\rangle &= \frac{(s_2-s_1)hk}{(h+k)}.
\end{align}

Now we have that $(x_{0},x_{0},\nabla_{x,y}\psi(x_{0},y_{0},w_{0}))=w_{0}\in \Delta$. Due to the transverse intersection of the graph of $(x,y)\to \nabla_{x,y}\psi(x,y,w_{0})$ at $w_{0}$, by the implicit function theorem, we have that given $w$ in a small enough conical neighborhood of $w_{0}$, there exists a small  neighborhood of $(x_{0},x_{0})$ and a unique $(x,y)$ in this neighborhood such that $(x,y,\nabla_{x,y}\psi(x,y,w))\in \Delta$. Now let $(s_{1},s_{2}, t, \theta_{1},\cdots,\theta_{n-1})$ be a critical point for the function $\wt{\psi}$. We have that 
\[
\g(t)+s_{1} \omega, \g(t)+s_{2}\om, \xi+(\g(t)+s_{1}\om-x_{0}) h(\xi,-\xi), \xi-(\g(t)+s_{2}\om-x_{0})k(\xi,-\xi)\in \Delta.
\]
This then implies that $s_{1}=s_{2}$. Substituting this into \eqref{Omega-xi0}, we get that $\langle \omega, \xi_{0}\rangle=0$.
Substituting $s= s_1=s_2$ in $\wt\psi_{\theta_1}$,$\wt\psi_{\theta_2},\dots,\wt\psi_{\theta_{n-1}} $ and in $s_1\wt\psi_{s_1}+s_2\wt\psi_{s_2}$, we get that $(\g(t)+s\omega-x_0)$ is orthogonal to each of the vectors $\omega$, $\omega_{\theta_1},\dots ,\omega_{\theta_{n-1}}$. But $\omega$, $\omega_{\theta_1},\omega_{\theta_2},\dots, \omega_{\theta_{n-2}}$ and $\omega_{\theta_{n-1}}$ are mutually orthogonal and therefore we have that $\g(t)+s\omega-x_0=0$.

Summarizing the above calculations, we have the following: 
Given a diagonal element, $w_{0}=(x_{0},\xi,x_{0},\xi)\in \Delta \setminus \Sigma$, at a critical point $(s,s,t,\theta_{1},\cdots,\theta_{n-1})$  of $\wt{\psi}$, the plane passing through $x_{0}$ and perpendicular to $\xi$ intersects the curve at $\g(t)$, and $x_{0}$ lies on the line passing through $\g(t)$ in the direction $\omega$ (which we have already determined is perpendicular to $\xi_{0}$).

The Hessian at the critical point is given by
\begin{equation*}
	\D^2_{s_1s_2t\theta_1\theta_2\dots\theta_{n-1}}\wt{\psi} = \begin{bmatrix}
		A & B\\
		B^t & C
	\end{bmatrix}
\end{equation*}
where
\begin{align*}
	&A = \begin{pmatrix}
		h & 0 & (\omega\cdot \g^\prime(t))h\\
		0 & k & (\omega\cdot \g^\prime(t))k\\
		(\omega\cdot \g^\prime(t))h & (\omega\cdot \g^\prime(t))k &|\g^\prime(t)|^2(h+k)
	\end{pmatrix}, \\\\
	&B^t = \begin{pmatrix}\xi_0\cdot\omega_{\theta_1} & -\xi_0\cdot\omega_{\theta_1} & (\g^\prime(t)\cdot\omega_{\theta_1})(h+k)s\\ \xi_0\cdot\omega_{\theta_2} & -\xi_0\cdot\omega_{\theta_2} & (\g^\prime(t)\cdot\omega_{\theta_2})(h+k)s\\\xi_0\cdot\omega_{\theta_3} & -\xi_0\cdot\omega_{\theta_3} & (\g^\prime(t)\cdot\omega_{\theta_3})(h+k)s\\ \vdots &\vdots &\vdots\\
	 \xi_0\cdot\omega_{\theta_{n-1}} & -\xi_0\cdot\omega_{\theta_{n-1}} & (\g^\prime(t)\cdot\omega_{\theta_{n-1}})(h+k)s
	\end{pmatrix},
\end{align*}
\begin{align*}
	&C = s^2(h+k)\begin{pmatrix}
		1 & 0 & 0 & \cdots & 0\\
		0 & (\sin\theta_1)^2 & 0 & \cdots & 0\\
		0 & 0 & (\sin\theta_1\sin\theta_2)^2 & \cdots& 0\\
		\vdots &\vdots &\vdots &\ddots&\vdots\\
		0 & 0 & 0 & \cdots & (\sin\theta_1\sin\theta_2\dots\sin\theta_{n-2})^2
	\end{pmatrix}.
\end{align*}
We have
\begin{align*}
	\begin{pmatrix}
		I_3 &-BC^{-1}\\
		0 & I_{n-1}
	\end{pmatrix}
	\begin{pmatrix}
		A & B\\
		B^t & C
	\end{pmatrix}
	\begin{pmatrix}
		I_3 & 0\\
		-C^{-1}B & I_{n-1}\\  
	\end{pmatrix}
	= \begin{pmatrix}
		A-BC^{-1}B^t & 0\\
		0 & C
	\end{pmatrix}.
\end{align*}
Then $\det(\D^2\wt{\psi})=\det(A-BC^{-1}B^t)\times \det(C)$ and sgn$(\D^2\wt{\psi})$= sgn$(A-BC^{-1}B^t)+$sgn$(C)$, where $\mbox{sgn}$ stands for the difference between the number of positive eigenvalues and the number of negative eigenvalues.
We have
\[
BC^{-1}B^t =\frac{1}{s^2(h+k)} \begin{pmatrix}
(n-1) &-(n-1) & Ks(h+k)\\
-(n-1) & (n-1) & -Ks(h+k)\\
Ks(h+k) & -Ks(h+k) & Ls^2(h+k)^2\\
\end{pmatrix},
\]
where $$K = \sum_{i=1}^{n-1}\left(\xi_0\cdot\frac{\omega_{\theta_i}}{\sin{\theta_1}\sin{\theta_2}\dots\sin{\theta_{i-1}}}\right)\left(\g^\prime(t)\cdot\frac{\omega_{\theta_i}}{\sin{\theta_1}\sin{\theta_2}\dots\sin{\theta_{i-1}}}\right)$$
and 
$$
L =\sum_{i=1}^{n-1}\left|\left(\g^\prime(t)\cdot \frac{\omega_{\theta_i}}{\sin{\theta_1}\sin{\theta_2}\ldots\sin{\theta_{i-1}}}\right)\right|^2.$$
Since $\xi_0\cdot\omega=0$, we have
\begin{align*}
	\xi_0 = \sum_{i=1}^{n-1}A_i \frac{\omega_{\theta_i}}{\sin{\theta_1}\sin{\theta_2}\ldots\sin{\theta_{i-1}}} \mbox{ for some} A^{i}, 1\leq i\leq n-1.
\end{align*}
Then
\[
 K  = \sum_{i=1}^{n-1} A_i\left(\g^\prime(t)\cdot\frac{\omega_{\theta_i}}{\sin{\theta_1}\sin{\theta_2}\dots \sin{\theta_{i-1}}}\right).
\]
Taking dot product with $ \g^\prime(t)$ with $\xi_0$ above, we have 
$$ K  = \g^\prime(t)\cdot\xi_0.$$
Now let us take 

$$F = \frac{n-1}{s^2(h+k)} \mbox{ and }  G = \frac{\g^\prime(t)\cdot \xi_0}{s}.$$
We then have
$$ BC^{-1}B^t = \begin{pmatrix}
F & -F & G\\
-F & F & -G\\
G & -G & L(h+k)
\end{pmatrix}$$
and 
$$A - BC^{-1}B^t = \begin{pmatrix}
h-F & F & (\g^\prime(t)\cdot\omega)h-G\\
F & k-F & (\g^\prime(t)\cdot\omega)k+G\\
(\g^\prime(t)\cdot\omega)h-G &(\g^\prime(t)\cdot\omega)k+G & H
\end{pmatrix}$$
where $ H = (| \g^\prime(t)|^2-L)(h+k) = |\g^\prime(t)\cdot \omega|^2(h+k)$.

Now 

\[
\text{ det }(A-BC^{-1}B^t) = -(h+k)G^2 = -\frac{h+k}{s^2}(\g^\prime(t)\cdot\xi_0)^2,
\]
and 
\[
 \det (C) = \prod_{k=1}^{n-2}(\sin^2\theta_k)^{n-1-k}.
 \]
Therefore 
\begin{align*}
	\left|\det (\D^2_{s_1,s_2,t,\theta_1,\theta_2,\dots,\theta_{n-1})})\right| &= \frac{h+k}{s^2} (\g^\prime(t)\cdot\xi_0)^2s^{2(n-1)}(h+k)^{n-1}\det (C)\\
	\\ &= (\g^\prime(t)\cdot\xi_0)^2(h+k)^n|\g(t)-x_0|^{2(n-2)}\det (C).
\end{align*}
From the above equality, we have that a critical point is non-degenerate if $\g^\prime(t)\cdot\xi_0\neq0$. Geometrically, this means that the plane $x+\xi^\perp$ intersects the curve transversely.	

Since $ \det (A-BC^{-1}B^t) = -\frac{h+k}{s^2}(\g^\prime(t)\cdot\xi_0)^2$, we have that this determinant is strictly negative due to the choice of $h$ and $k$. Therefore $\mbox{sgn}(A-BC^{-1}B^t)$ is either $1$ or $-3$. Consider 
$$\begin{bmatrix}
0 & 0 & 1
\end{bmatrix}(A-BC^{-1}B^t) \begin{bmatrix}
0 \\ 0 \\ 1
\end{bmatrix} = H  =|a^\prime(t)\cdot \omega|^2(h+k) >0.$$ 
Thus $A-BC^{-1}B^t$ has at least one positive eigenvalue which implies  sgn$(A-BC^{-1}B^t)$ is $1$. Combining $\mbox{sgn}(A-BC^{-1}B^t)$ and sgn$(C)$, we conclude that the $\mbox{sgn}(\D^2\tilde\psi)$ is $n$.

Using the method of stationary phase, the leading term in the stationary phase expansion of $e^{\I\lambda\tilde\psi}(A(\chi e^{i\lambda\psi})$ as $\lambda\to \infty$ is
\begin{align*}
	\frac{(2\pi)^\frac{n+2}{2} e^\frac{n\pi i}    {4}\phi_1(x_0)\phi_2(x_0)\omega^{i_1}(t_q)\cdots\omega^{i_m}(t_q)\omega^{j_1}(t_q)\cdots\omega^{j_m}(t_q)}{\lambda^\frac{n+2}{2}(h+k)^\frac{n}{2}|(\g^\prime(t_{q})\cdot\xi_0)||(\g(t_q)- x_0)|^{n-2}}.
\end{align*}
Here $\{\g_{t_{q}}\}$ are the intersection points of the curve $\g$ with the plane $x+\xi^{\perp}$. Note that for each intersection point $t_{q}$ on the curve $\g$, the $\det(C)$ term in the Hessian cancels with the $\det(C)$ expression in the integrand.

The principal symbol $\sigma_{0}(x_{0},\xi,x_{0},\xi)$ (which by abuse of notation, we will denote it by $\sigma_{0}(x_{0},\xi)$) is the sum of the expressions arising from all the critical points of the phase function. We have 
\begin{align}\label{eq:a}
	\sigma_0(x_{0},\xi)= \sum_q  \frac{(2\pi)^\frac{n+2}{2} e^\frac{n\pi i}    {4}\phi_1(x_0)\phi_2(x_0)\omega^{i_1}(t_q)\cdots\omega^{i_m}(t_q)\omega^{j_1}(t_q)\cdots\omega^{j_m}(t_q)}{\lambda^\frac{n+2}{2}(h(\xi_0,-\xi_0)+k((\xi_0,-\xi_0))^\frac{n}{2}|(\g^\prime(t_q(\xi_0))\cdot\xi_0)||(\g(t_q(\xi_0))- x_0)|^{n-2}}.
\end{align}
Note that due to the condition imposed on the closed conic set $K$ in the statement of Theorem \ref{Main theorem}, we have that the intersection points are all transverse intersection points.

The expression for the principal symbol above is dependent on the choice of $\psi$ (due to its dependence on $h$ and $k$), and the cutoff functions $\phi_1$ and $\phi_2$. To make it independent of the choice of these functions, we divide \eqref{eq:a} by the principal symbol of the identity operator obtained by the same cutoff functions and $\psi$. The principal symbol of the identity operator can be computed analogously  by  computing the  leading order term in the following integral as $\lambda\rightarrow\infty$
\begin{align}\label{Identity operator}
	\langle \delta_\Delta,e^{-i\psi(x,y,\omega_0)}\phi_1(x)\phi_2(y)\rangle = \int\phi_1(x)\phi_2(x)e^{-i\lambda\frac{1}{2}|x-x_0|^2(h(\xi_0,-\xi_0)+k((\xi_0,-\xi_0))}\D x.
\end{align}
We have that $x = x_0$ is the only critical point for phase function $\frac{1}{2}|x-x_0|^2(h(\xi_0,-\xi_0)+k((\xi_0,-\xi_0))$ and Hessian at the critical point is $(h+k)I_n$. Hence the principal symbol of \eqref{Identity operator} is
\begin{align*}
	\frac{(2\pi)^\frac{n}{2} e^\frac{n\pi i}{4}\phi_1(x_0)\phi_2(x_0)}{\lambda^\frac{n}{2}(h+k)^\frac{n}{2}}.
\end{align*}
Dividing \eqref{eq:a} by the principal symbol of identity operator above, we get the principal symbol $A_{0}(x_{0},\xi,x_{0},\xi)$ (which we will denote by $A_{0}(x_{0},\xi)$) of $\Rc^{*}_{\g}\omega^{i_1}\cdots\omega^{i_m}\omega^{j_1}\cdots\omega^{j_m} \Rc_{\g}$ (here $\Rc_{\g}$ is the restricted ray transform on functions) on $\Delta\setminus \Sigma$ to be
\begin{align*}
	A_0 (x_{0},\xi)_= \sum_q  \frac{2\pi\omega^{i_1}(t_q)\cdots\omega^{i_m}(t_q)\omega^{j_1}(t_q)\cdots\omega^{j_m}(t_q)}{|\xi||\g^\prime(t_q(\xi_0)\cdot\xi_0)||\g(t_q(\xi_0)- x_0)|^{n-2}}.
\end{align*}
In the next section, we will prove ellipticity of this symbol on solenoidal tensor fields.
\section{Ellipticity of the symbol on the solenoidal part}\label{sect:ellipticity}
\begin{proposition}\label{Ellipticity Proposition}
	Let $f$ be a solenoidal symmetric $m$-tensor field in $\mathbb{R}^n$ in the sense that $\xi^{i_{m}}f_{i_{1}\cdots i_{m}}(\xi)=0$. Suppose $A_0 f =0$ then $f=0$.
\end{proposition}

We will need the following straightforward lemmas for Proposition \ref{Ellipticity Proposition}.
\begin{lemma}\label{1.1}
	If $\omega_{1},\cdots, \omega_{n}$ are $n$-linearly independent vectors in $\mathbb{R}^n$ and 
	\begin{equation}\label{eq:21}
	C_1\otimes\omega_{1}+C_2\otimes\omega_{2}+ \dots+C_n\otimes \omega_{n} =0
	\end{equation}
	($\otimes$ denotes the usual tensor product), $C_i \in \mathbb{R}^p$, for $i = 1,\dots, n$ and for an arbitrary $p$. Then each $C_i$ is the $0$-vector.
\end{lemma}
\begin{proof}
	Let $C_i = (c_{ij})_{j=1}^p \in \mathbb{R}^p$. Now \eqref{eq:21} can be written as
	\begin{align}\label{eq:22}
		\begin{pmatrix}
			c_{11}\\
			c_{12}\\
			\vdots\\
			c_{1p}\\
		\end{pmatrix}\otimes \omega(t_1)+  \begin{pmatrix}
		c_{21}\\
		c_{22}\\
		\vdots\\
		c_{2p}\\
	\end{pmatrix}\otimes \omega(t_2)+ \cdots +\begin{pmatrix}
	c_{n1}\\
	c_{n2}\\
	\vdots\\
	c_{np}\\
\end{pmatrix}\otimes\xi=0.
\end{align}
From the above equation, we get 
\[ c_{1j}\omega_{1}+ c_{2j}\omega_{2}+\cdots+ c_{nj}\omega_{n}=0 \text{ for } j=1,\dots, p.
\] 
which implies that $c_{ij}=0$ for $i=1,\cdots,n$ and $j=1,\dots,p$. 

\end{proof}
The next statement is a standard fact, but we prefer to give a quick proof as it follows immediately from the previous lemma.
\begin{lemma}
	Let $\omega_{1},\cdots, \omega_{n}$ be linearly independent vectors in $\Rb^{n}$. For any $m$, the collection of $n^{m}$ $m$-tensors of the form 
	\[
	\omega_{i_{1}}\otimes \cdots\otimes \omega_{i_{m}} \mbox{ with } 1\leq i_{1},\cdots,i_{m}\leq n,
	\]
	is linearly independent as well.
\end{lemma}

\bpr
We use mathematical induction on $m$. For $m=1$, there are $n$ vectors $\omega_{1},\cdots,\omega_{n}$, and these are linearly independent by assumption. Now assume the collection of $k-1$ tensors $\omega_{i_{1}}\otimes\cdots \otimes \omega_{i_{k-1}}$ are linearly independent. 
Assume that the following linear combination is $0$: 
\[
c_{i_{1}\cdots i_{k}}\omega_{i_{1}}\otimes \cdots \otimes \omega_{i_{k}}=0.
\] 
We can write this as 
\[
C_{1}\otimes \omega_{1}+\cdots C_{n}\otimes \omega_{n}=0,
\]
where each $C_{i}$ is a linear combination involving the coefficients $c_{i_{1}\cdots i_{k}}$ of $k-1$ tensors of the form $\omega_{i_{1}}\otimes\cdots \otimes \omega_{i_{k-1}}$.
By the previous lemma each $C_{i}=0$, and then the result follows from the induction hypothesis.
\epr
\begin{lemma}\label{lemma: multinomial independence for symmetric product}
	Let $\omega_{1},\cdots, \omega_{n}$ be linearly independent vectors in $\Rb^{n}$. For any $m$, the collection of $n+m-1\choose m$ $m$-tensors of the form 
	\[
	\omega_{i_{1}}\odot \cdots\odot \omega_{i_{m}} \mbox{ with } 1\leq i_{1},\cdots,i_{m}\leq n,
	\]
	where, recall that $\odot$ denotes symmetric tensor product, is linearly independent as well.
\end{lemma}

\bpr
We have 
\[
\omega_{i_{1}}\odot \cdots \odot \omega_{i_{m}}=\frac{1}{m!}\sum\limits_{\sigma\in S_{m}} \omega_{i_{\sigma(1)}}\otimes \cdots \otimes \omega_{i_{\sigma(m)}}.
\]
A linear combination of the $m$-tensors $\omega_{i_{1}}\odot \cdots\odot \omega_{i_{m}}$ is a linear combination of the tensors $\omega_{i_{\sigma(1)}}\otimes \cdots \otimes \omega_{i_{\sigma(m)}}$. Now the result follows from the previous lemma.
\epr

\begin{proof}[Proof of Proposition \ref{Ellipticity Proposition}]
	A symmetric $m$-tensor field in $\mathbb{R}^n$ has $n+m-1\choose m$ distinct components. Therefore we need ${n+m-1}\choose m$ conditions in order to determine a symmetric $m$-tensor field uniquely in $\mathbb{R}^n$. From $A_{0}f=0$, we have 
	\begin{align*}
		\sum_q  \frac{2\pi\omega^{i_1}(t_q)\cdots\omega^{i_m}(t_q)\omega^{j_1}(t_q)\cdots\omega^{j_m}(t_q)f_{i_1\dots i_m}}{|\xi||(\g^\prime(t_q(\xi_0))\cdot\xi_0)(\g(t_q(\xi_0))- x_0)|}&=0
		\end{align*}
		Multiplying by $f_{j_{1}\cdots j_{m}}$ and adding, we get
		\begin{align*}
		\sum_q  \frac{2\pi\omega^{i_1}(t_q)\cdots\omega^{i_m}(t_q)\omega^{j_1}(t_q)\cdots\omega^{j_m}(t_q)f_{i_1\dots i_m}f_{j_1\dots j_m}}{|\xi||(\g^\prime(t_q(\xi_0))\cdot\xi_0)(\g(t_q(\xi_0))- x_0)|}&=0.
		\end{align*}
		This gives
		\begin{align*}
		\sum_q  \frac{2\pi(\omega^{i_1}(t_q)\cdots\omega^{i_m}(t_q)f_{i_1\dots i_m})^2}{|\xi||(\g^\prime(t_q(\xi_0))\cdot\xi_0)(\g(t_q(\xi_0))- x_0)|}&=0.
	\end{align*}
	Hence 
	\begin{align}\label{eq:conditions due to symbol}
		\omega^{i_1}(t_q)\cdots\omega^{i_m}(t_q)f_{i_1\dots i_m}=0, \text{ for } q= 1,\dots, {n+m-2\choose m}.
	\end{align} 
Using the fact that 
	\begin{align}\label{eq:conditions due to solenoidal in ndim}
		\xi^{i_m}f_{i_1\dots i_m}=0,
	\end{align} 
	we have that 
	\begin{align}\label{eq:modified conditions due to solenoidal}
		\omega^{i_1}(t_{j_1})\dots\omega^{i_{m-1}}(t_{j_{m-1}})
		\xi^{i_m}f_{i_1\dots i_m}=0
	\end{align}
	for $1\leq j_1, \cdots, j_{m-1} \leq n$ with $\omega(t_{n})=\xi$. Since $f_{i_{1}\cdots i_{m}}$ is symmetric, \eqref{eq:modified conditions due to solenoidal} can be written as 
	\[
	\lb \omega(t_{j_{1}})\odot \cdots \odot \omega(t_{j_{m-1}})\odot \xi \rb^{i_{1}\cdots i_{m}} f_{i_{1}\cdots i_{m}}=0.
	\]
We now define the following ${n+m-1\choose m} \times {{n+m-1}\choose m}$ matrix: 
	\[A =\begin{pmatrix}
	\omega(t_1)^{\odot m}, & \cdots & \omega(t_{L(n,m)})^{\odot m}, & \omega(t_{j_1})\odot\dots \odot \omega(t_{j_{m-1}})\odot\xi
	\end{pmatrix}
\]
\
for 	$1\leq j_1,\cdots, j_{m-1} \leq n$.

Then \eqref{eq:conditions due to symbol} and \eqref{eq:conditions due to solenoidal in ndim} together can be written as 
	\begin{align}\label{eq:combined conditions}
		A^t f =0.
	\end{align}
	Showing $f=0$ is equivalent to proving that  $A$ has full rank. Suppose that a linear combination of the columns of $A$ is equal to $0$: 
	\begin{align}\label{eq:linear combination of columns of A}
		\sum_{i=1}^{L(n,m)}c_i\omega(t_i)^{\odot m}+\sum_{1\leq j_1, \cdots, j_{m-1} \leq n} c_{j_1\dots j_{m-1}}\omega(t_{j_1})\odot\cdots \odot \omega(t_{j_{m-1}})\odot \xi =0.
	\end{align}
	Note that in \eqref{eq:linear combination of columns of A}, we make the assumption that $\omega(t_{n})=\xi$. Now 
	\begin{align*}
		&c_1\omega(t_1)^{\odot m} +\dots + c_{n-1}\omega(t_{n-1})^{\odot m}+
		\sum_{i=n}^{L(n,m)}c_i\left(\sum_{j=1}^{n-1}\alpha_{ij}\omega(t_j)\right)^{\odot m}\\ 
		&\quad \quad \quad \quad \quad \quad \quad \quad \quad \quad+\sum_{1\leq j_1, \cdots, j_{m-1} \leq n} c_{j_1\dots j_{m-1}}\omega(t_{j_1})\odot\cdots \odot \omega(t_{j_{m-1}})\odot \xi =0.
		\end{align*}
		We can write the above equation as 
		\begin{align*}
		&\sum_{1\leq j_1, \cdots, j_{m} \leq n-1} d_{j_1\dots j_{m-1}}\omega(t_{j_1})\odot\dots \odot \omega(t_{j_{m}}) \\ 
		&\quad \quad \quad \quad \quad \quad \quad \quad \quad \quad +
		\sum_{1\leq j_1, \cdots, j_{m-1} \leq n} c_{j_1\dots j_{m-1}}\omega(t_{j_1})\odot\cdots \odot \omega(t_{j_{m-1}})\odot \xi =0.
	\end{align*}
	Now from Lemma \ref{lemma: multinomial independence for symmetric product}, we have that $c_{j_1\dots j_{m-1}}$ and $d_{j_1\dots j_{m-1}}$ are $0$. Letting  $c_{j_1\dots j_{m-1}}=0$ in  \eqref{eq:linear combination of columns of A}, we get
	$$	\sum_{i=1}^{L(n,m)}c_i\omega(t_i)^{\odot m} =0.$$
	Now by the Kirillov-Tuy condition, each $c_{i}=0$. Hence we have that the matrix $A$ has full rank. Finally, we have that $f=0$. The proof of Proposition \ref{Ellipticity Proposition} is complete.
	\epr

\section{Microlocal Inversion}\label{sect:microlocal_inversion}
In this section, we will construct a microlocal parametrix for the operator $\Rc_{\g}^{*}\Rc_{\g}$. This will complete the proof of Theorem \ref{Main theorem}.

\bpr[Proof of Theorem \ref{Main theorem}]

Using a lexicographic ordering, let us view the entries of the symmetric $m$-tensor $\sqrt{\frac{2\pi}{|\xi||\g'(t_{s}(\xi_{0})\cdot \xi_{0}(\g(t_{s}(\xi_{0})-x_{0})|}}\omega(t_{s})^{\odot m}$ as a column vector $v_s$, and with that define the matrix  
\[ 
V=\begin{pmatrix}
v_{1}& \cdots &v_{L(n,m)}&\cdots v_{N}
\end{pmatrix}.
\]
Here $N$ is the number of intersection points of the plane $x_{0}+\xi^{\perp}$ with the curve $\g$. 

We then have that 
\[
A_{0}(x,\xi)= V V^{t}.
\]

Again using a lexicographic ordering, $A_{0}$ can be  viewed as a matrix. The entries of $V$ are real, and since the rank of $V$ is $L(n,m)$, the rank of $A_{0}$ is $L(n,m)$ as well.  Since $A_0$ is symmetric, using the spectral decomposition theorem, we have that $A_0=P DP^{t}$, where $P$ is the orthogonal matrix comprising the eigenvectors of $A_0$ and $D$ is the diagonal matrix of eigenvalues of $A_0$. 
This implies that there are $L(n,m)$ non-zero entries in the diagonal matrix $D$. Now let  $D^{-}$ be the diagonal matrix obtained by replacing the $L(n,m)$ nonzero entries of $D$  by its reciprocals and define $B_0(x,\xi) = \sigma(\mathcal{S})PD^{-}P^t$ where $\mathcal{S}$ is solenoidal projection of a symmetric tensor field, and $\sigma(\Sc)$ is its principal symbol. 

Next, we have 
\begin{align*}
	B_0(x,\xi)A_0(x,\xi)&= \sigma(\mathcal{S})P \begin{pmatrix}
		I_{L(n,m)} & 0 \\
		0 &  0\\ 
	\end{pmatrix} P^t\\
	&= \sigma(\mathcal{S}).
\end{align*}
In the last line we used the fact that all the columns in  $P$ after $L(n,m)$-th column are the eigenvectors corresponding to the $0$ eigenvalue. Thus $ B_0(x,\xi)A_0(x,\xi)$ equals $\sigma(\mathcal{S})$. Since the singularities of $P$ and $D^{-}$ are possibly only on $\Sigma$, we have that the  entries of $B_0(x,\xi)$ lie in the symbol class $I(\Delta,\Lambda)$. Let $\mathcal{B}_0$ be the operator with symbol matrix given by 
$$b_0(x,\xi) = \left \{
\begin{array}{ll}
B_0(x,\xi)  & \mbox{ if }(x,\xi)\in \Xi_0 \\
0 & \mbox{ otherwise. }
\end{array}
\right. $$
We have that the solenoidal projection $\Sc$ is of order $0$ and $\mathcal{R}_{\g}^{*}\mathcal{R}_{\g} \in I^{-1,0}(\Delta,\Lambda)$ \cite{Greenleaf-Uhlmann-Duke1989}. The principal symbol of $\Bc_{0}\Rc^{*}_{\g}\Rc_{\g}$ is $\sigma(\Sc)$. Therefore $\Bc_{0}\Rc^{*}_{\g}\Rc_{\g}\in I^{-\frac{1}{2}, \frac{1}{2}}$. Indeed by the properties of $I^{p,l}$ classes, on $\Delta\setminus \Sigma$, an $I^{p,l}$ distribution is a pseudodifferential operator of order $p+l$. Now since $\Sc$ is of order $0$, we have that $p+l=0$. Also we know from \cite{Guillemin-Uhlmann} that the principal symbol of an $I^{p,l}$  distribution has a singularity along $\Sigma$ as a conormal distribution of order $l-k/2$. In our case $k=1$, and away from $\Delta\setminus \Sigma$, $\Sc$ is a classical pseudodifferential operator. Therefore $l=\frac{1}{2}$. Now $p=-\frac{1}{2}$. 

Using the symbol calculus for $I^{p,l}(\Delta,\Lambda)$, we can decompose $R_1 = \mathcal{B}_0\mathcal{R}_{\g}^{*}\mathcal{R}_{\g}-\mathcal{S}$  as $R_{11}+R_{12}$ where as $R_{11}\in I^{-\frac{3}{2},\frac{1}{2}}$ and $R_{12} \in I^{-\frac{1}{2},-\frac{1}{2}}$.

Now let $\Pc$ denote the potential projection of a symmetric $m$-tensor field. We have the following lemma.
\begin{lemma}\label{Principal symbol remainders}
	The principal symbols of $R_{1j}\circ \Pc$ as $I^{-\frac{3}{2},\frac{1}{2}}$ and $I^{-\frac{1}{2},-\frac{1}{2}}$ distributions respectively for $j= 1,2$ are $0$.
\end{lemma}
\begin{proof} We first consider $R_{1}\circ \Pc$. We have 
	\begin{align*}
		\lb R_1\circ \Pc \rb  f &= \lb (\mathcal{B}_0\mathcal{R}_{\g}^{*}\mathcal{R}_{\g}-
		\mathcal{S})\circ \Pc \rb f =0.
	\end{align*}
	This follows from the fact that the solenoidal projection of a potential field is $0$ and $\lb\Rc_{\g}\circ \Pc\rb  f=0$. Now 
		\[
		\lb R_{1}\circ \Pc \rb f= \lb (R_{11}+R_{12})\circ \Pc \rb f=0.
		\]
		From this, we have
		\[
		\lb R_{11}\circ \Pc \rb f = -\lb R_{12}\circ \Pc \rb f.
	\]
	This then implies that $R_{11}\circ \Pc$ and $R_{12}\circ \Pc$ both belong to $I^{-\frac{3}{2},-\frac{1}{2}}$. Therefore the principal symbol of $R_{11}\circ \Pc$ as an $I^{-\frac{3}{2},\frac{1}{2}}$ distribution is $0$ and that of $R_{12}\circ \Pc$ as an $I^{-\frac{1}{2},-\frac{1}{2}}$ distribution is $0$ as well.

\end{proof}

Let $\NS(A), \RS(A)$ and $\CS(A)$  denote the null space, row space and column space, respectively of a matrix $A$. We have the following standard result from linear algebra.
\begin{lemma}\label{row space lemma}
	If $A$ and $B$ are two matrices of the same order such that $\NS(A)\subset \NS(B)$, then $\RS(B)\subset \RS(A)$. 
\end{lemma} 
Now from Lemma \ref{Principal symbol remainders}, we have that $\NS\lb\sigma_{0}(\mathcal{S})\rb\subset 
\NS\lb \sigma_{0}(R_{1j}\rb)$ for $j=1,2$.  Then by Lemma \ref{row space lemma}, we have that $\RS\lb\sigma_0(R_{1j})\rb\subset \RS\lb \sigma_{0}(\Sc)\rb$.  is contained in row space of $\sigma(\mathcal{S})$ for $j=1,2$. Now by Proposition \ref{Ellipticity Proposition}, we have that $\RS(\sigma_{0}(\Sc))=\RS(A_{0})$. 
Hence there exist matrices $r_{1j}$ such that $\sigma_0(R_{1j}) = r_{1j}A_0$ for $j =1,2$. 

Let $\Bc_{11}$ and $\Bc_{12}$ be operators having symbols $-r_{11}$ and $-r_{12}$ respectively. Let $\Bc_{1}=\Bc_{11}+\Bc_{12}$, and define $R_{2}=(\Bc_0+\Bc_1)\Rc_{\g}^{*}\Rc_{\g}-\Sc$. We can rewrite $R_{2}$ as 
\begin{align*}
R_{2}&=\Bc_{11}\Rc_{\g}^{*}\Rc_{\g}+\Bc_{12}\Rc_{\g}^{*}\Rc_{\g} +\Bc_0 \Rc_{\g}^{*}\Rc_{\g}-\Sc\\
&=\Bc_{11}\Rc_{\g}^{*}\Rc_{\g}+R_{11}+\Bc_{12}\Rc_{\g}^{*}\Rc_{\g}+R_{12}.
\end{align*}

Let us denote $K_1=\Bc_{11}\Rc_{\g}^{*}\Rc_{\g}+R_{11}$ and $K_2=\Bc_{12}\Rc_{\g}^{*}\Rc_{\g}+R_{12}$. We have $K_1\in I^{-\frac{3}{2},\frac{1}{2}}$ and $K_2\in I^{-\frac{1}{2},-\frac{1}{2}}$. By construction, the principal symbols of $K_1$ and $K_2$ are $0$. Hence using the symbol calculus, we decompose 
\begin{align*}
& K_1=K_{11}+K_{12}, \quad K_{11}\in I^{-\frac{5}{2},\frac{1}{2}}, K_{12}\in I^{-\frac{3}{2},-\frac{1}{2}}\\
&K_2=K_{21}+K_{22}, \quad K_{21}\in I^{-\frac{3}{2},-\frac{1}{2}}, K_{12}\in I^{-\frac{1}{2},-\frac{3}{2}}.
\end{align*}
With this, we can write 
\[
R_2=\underbrace{K_{11}}_{R_{20}}+\underbrace{K_{12}+K_{21}}_{R_{21}}+
	\underbrace{K_{22}}_{R_{22}},
\]
where $R_{20}\in I^{_\frac{5}{2}}$, $R_{21}\in I^{-\frac{3}{2},-\frac{1}{2}}$ and $R_{22}\in I^{-\frac{1}{2},-\frac{3}{2}}$. Therefore 
\[
R_{2}\in \sum\limits_{j=0}^{2} I^{-\frac{1}{2}-2+j,\frac{1}{2}-j}.
\] 
Proceeding recursively, we get a sequence of operators $$ R_N \in \sum_{j=o}^N I^{-\frac{1}{2}-N+j,\frac{1}{2}-j}.$$
We write this as 
\[
R_N\in \sum_{j=0}^{[\frac{N}{2}]}I^{-\frac{1}{2}-N+j,\frac{1}{2}-j}+ 
\sum_{j=[\frac{N}{2}]+1}^NI^{-\frac{1}{2}-N+j,\frac{1}{2}-j}.
\]

In the first term above, since $-\frac{1}{2}-N+j\leq -\frac{1}{2}-N+[\frac{N}{2}]$ and $\frac{1}{2}-j\leq \frac{1}{2}$, we have that first term belongs to $I^{-\frac{1}{2}-N+[\frac{N}{2}],\frac{1}{2}}$. Letting $N\to \infty$, we have that this term is smoothing, since  $\cap_{p} I^{p,l}\subset C^{\infty}$. For the second term, we use the fact that for $[\frac{N}{2}]+1 \leq j\leq N$, 
$-\frac{1}{2}-N+j\leq -\frac{1}{2}$ and $\frac{1}{2}-j\leq -\frac{1}{2}-[\frac{N}{2}]$. Again letting $N\to \infty$, since $\cap_{l} I^{p,l}(\Delta,\Lambda)\subset I^{p}(\Lambda)$, we have that the second as $N\to \infty$ is an operator in $I^{-\frac{1}{2}}(\Lambda)$ which we denote by $\Ac$. Therefore defining an operator $\Bc=\mathcal{B}_0+\mathcal{B}_1+\cdots$, we have for each $x\in \pi_{1}(K)$,
\[
 \mathcal{B}\mathcal{R}_{\g}^{*}\mathcal{R}_{\g}f(x) = f^{s}(x)+\mathcal{A}f(x)+\mbox{smoothing terms}.
\]
This completes the proof of Theorem \ref{Main theorem}.
\epr

\section{Some examples}\label{sect:examples}
In this final section, we justify the remarks that we made in Section \ref{sect:Statement_main_result} about the Kirillov-Tuy condition.
\subsection{Case $n-1=2$}
	\begin{prop}
		If $\omega_{1}, \cdots , \omega_{m+1}$  are $(m+1)$-vectors contained in a plane of $\Rb^3$ such  that any two of them are linearly independent.  Then 
		the  following matrix has full rank: $$ A_m =\begin{pmatrix}
		\omega_{1}^{\odot m} & . & . & . & \omega_{m+1}^{\odot m}
		\end{pmatrix}.$$
	\end{prop} 
	\begin{proof}	
Define $$ \tilde{A}_m =\begin{pmatrix}
		\omega_{1}^{\otimes m} & . & . & . & \omega_{m+1}^{\otimes m}
		\end{pmatrix}.$$ 
Observe $\tilde{A}_m$ has $A_m$  as $(m+1) \times (m+1)$ block (by doing some row operations). Therefore it is sufficient to prove that $\tilde{A}_m$ has full rank.
We prove this using induction on $m$. Clearly the result is true for $ m=1$. Let us assume that  the matrix $\tilde{A}_k$ (for $m=k$) has full rank, that is, the vectors 
	$$ \omega_{1}^{\otimes k} , \cdots, \omega_{k+1}^{\otimes k}$$
	are linearly independent.
	Using this, we want to prove that the vectors 
	$$ \omega_{1}^{\otimes (k+1)} , \cdots, \omega_{k+2}^{\otimes (k+1)}$$
	are also linearly independent.

	Consider
	\begin{align}\label{eq:11}
	\sum_{i=1}^{k+2} c_i\omega_{i}^{\otimes(k+1)} =0.
	\end{align} 
 Using the pairwise independent condition, we write each $\omega_{i} =\alpha_i\omega_{1}+\beta_i\omega_{2}$ for $i\geq 3$ with the conditions that $\alpha_i\neq 0 \neq \beta_i$ for any $i\geq 3$ and $ \alpha_i\beta_j-\alpha_j\beta_i\neq0$ for $i\neq j$. From \eqref{eq:11}, we have
	\begin{align*}
	\left(c_1\omega_{1}^{\otimes k} +\sum_{i=3}^{k+2} \alpha_i c_i \omega_{i}^{\otimes k}\right)\otimes \omega_{1}+ \left(c_2\omega_{2}^{\otimes k} +\sum_{i=3}^{k+2} \beta_i c_i \omega_{i}^{\otimes k}\right)\otimes \omega_{2} &=0.
	\end{align*}
	Since $\omega_{1}$ and $\omega_{2}$ are independent, from Lemma \ref{1.1} , we get 
	\begin{align}\label{eq:12}
	c_1\omega_{1}^{\otimes k} +\sum_{i=3}^{k+2} \alpha_ic_i\omega_{i}^{\otimes k}=0,\\\label{eq:13}
	c_2\omega_{2}^{\otimes k} +\sum_{i=3}^{k+2} \beta_ic_i\omega_{i}^{\otimes k}=0.
	\end{align}
	If $c_{k+2} = 0$ then we are done and if $c_{k+2} \neq 0$ then we can use \eqref{eq:12} and \eqref{eq:13} to get the following:
\begin{align*}
\omega_{k+2}^{\otimes k} = -\frac{1}{\alpha_{k+2}c_{k+2}}\left( c_1 \omega_{1}^{\otimes k} + \sum_{i=3}^{k+1} \alpha_ic_i\omega_{i}^{\otimes k}\right)
\end{align*}
and 
\begin{align*}
\omega_{k+2}^{\otimes k} = -\frac{1}{\beta_{k+2}c_{k+2}}\left( c_2 \omega_{2}^{\otimes k} + \sum_{i=3}^{k+1} \beta_ic_i\omega_{i}^{\otimes k}\right).
\end{align*}
Equating the above two equations, we get 
\begin{align*}
\frac{c_1}{\alpha_{k+2}}\omega_{1}^{\otimes k} -\frac{c_2}{\beta_{k+2}}\omega_{2}^{\otimes k} +\sum_{i=3}^{k+1} c_i\left(\frac{\alpha_i\beta_{k+2} -\beta_i\alpha_{k+2} }{\alpha_{k+2}\beta_{k+2}}\right) \omega_{i}^{\otimes k} = 0.
\end{align*} 
This implies $c_i = 0$ for $i=1, \dots , k+1$ because of the pairwise independent condition and the induction hypothesis. Letting  $c_i=0$ for $i= 1, \dots, k+1$ in equation \eqref{eq:12}, we have
$$ \alpha_{k+2}c_{k+2} \omega_{k+2}^{\otimes k} = 0,$$
which implies $c_{k+2}=0$ because $\alpha_{k+2} \neq 0$ and $\omega_{k+2}$ is a non-zero vector.

	\end{proof}
\subsection{Case $n-1=3$ and $m=2$}
In this case we consider a collection of six vectors $\omega_{1}, \omega_{2}, \cdots, \omega_{6}$ in $\Rb^{3}$ (since in the Kirillov-Tuy condition, the vectors are restricted to a hyperplane) with the condition that any three of them are independent. 
\begin{lemma}
	If $\omega_{1}, \omega_{2}, \cdots, \omega_{6}$ be six vectors in $\Rb^{3}$ such that any three of them are linearly independent. Then the rank of the following matrix 
	$$\wt{V}=  \begin{pmatrix}
	\omega_{1}^{\otimes 2} & \omega_{2}^{\otimes 2} & \cdots & \omega_{6}^{\otimes 2}
	\end{pmatrix}$$
	is $6$.
\end{lemma}
\begin{proof}We can write 
	\[
	\omega_{i} = \alpha_i\omega_{1}+\beta_i\omega_{2}+\g_i\omega_{3}\text{, for } i = 4 ,5 ,6,
	\]
	and we have that the matrix 
	\[\begin{pmatrix}
	\alpha_4 & \beta_4 & \g_4\\
	\alpha_5 & \beta_5 & \g_5\\
	\alpha_6 & \beta_6 & \g_6
	\end{pmatrix}
	\] is invertible. 
	Now it is enough to show that the matrix
	\[
	\begin{pmatrix}
		1 & 0 & 0 & 0 & 0 & 0\\
		0 & 1 & 0 & 0 & 0 & 0\\
		0 & 0 & 1 & 0 & 0 & 0\\
		0 & 0 & 0 & \alpha_4\beta_4 & a_4\beta_4 & \alpha_4a_4\\
		0 & 0 & 0 & \alpha_5\beta_5 & a_5\beta_5 & \alpha_5a_5\\
		0 & 0 & 0 & \alpha_6\beta_6 & a_6\beta_6 & \alpha_6a_6\\
	\end{pmatrix} 
	\]
	is invertible.
	 We can write 
	$$ \begin{pmatrix}
	\alpha_4\beta_4 & \g_4\beta_4 & \alpha_4\g_4\\
	\alpha_5\beta_5 & \g_5\beta_5 & \alpha_5\g_5\\
	\alpha_6\beta_6 & \g_6\beta_6 & \alpha_6\g_6
	\end{pmatrix} = \begin{pmatrix}
	\alpha_4 & \beta_4 & \g_4\\
	\alpha_5 & \beta_5 & \g_5\\
	\alpha_6 & \beta_6 & \g_6
	\end{pmatrix}\times_{H} \begin{pmatrix}
	\beta_4 & \g_4 & \alpha_4 \\
	\beta_5 & \g_5 & \alpha_5\\
	\beta_6 & \g_6 & \alpha_6
	\end{pmatrix} $$
	where  $(A\times_{H} B)_{ij} = A_{ij}\cdot B_{ij}$ is the Hadamard product of matrices. It is well-known that 
	det($ A \times_{H} B) \geq $ det$(A)$det$(B$). Since both matrices on the right hand side are non-singular, the rank of $\wt{V}$ is $6$.
\end{proof}
\subsection{Counter example for $n-1=3$ and $m=3$} We give an example for the case of (3) in Remark \ref{Examples-remark}.
\begin{example}
	Consider the following collection of $L(4,3) = 10$ vectors: 
	\begin{align*}
		& v_1 = (1,0,0,0),v_2 =(0,1,0,0),v_3 =(0,0,1,0),\\
		& v_i = (1, p_i, p_i^2,0) 
						\end{align*}
			for $i= 4, \cdots, 10$ and $p_i$'s are distinct primes.
	 Then any three vectors of them are linearly independent. And it is easy to check that vectors $\{v_i^{\odot 3}\}_{i=1}^{10}$ are not independent.
\end{example}
{\bf Acknowledgments: } The authors are indebted to  Vladmir Sharafutdinov and Eric Todd Quinto for several  fruitful discussions regarding this work.

VK was supported in part by NSF grant DMS 1616564. Both authors benefited from the support of Airbus Corporate Foundation Chair grant titled ``Mathematics of Complex Systems'' established at TIFR CAM and TIFR ICTS, Bangalore, India
\bibliographystyle{plain}
\bibliography{reference}
\end{document}